\documentclass{amsart}
\usepackage{amssymb,enumerate,bm,color,here,url,hyperref}
\usepackage{tikz}
\usetikzlibrary{arrows.meta}

\usepackage{graphicx} 

\theoremstyle{plain}
	\newtheorem{thm}{Theorem}[section]
	\newtheorem{lem}[thm]{Lemma}
	\newtheorem{prop}[thm]{Proposition}

	\newtheorem{ex}[thm]{Example}
\theoremstyle{definition}
	\newtheorem{dfn}[thm]{Definition}
\theoremstyle{remark}

\DeclareMathOperator{\Isom}{Isom}

\DeclareMathOperator{\lk}{lk}

\newcommand{\bbZ}{\mathbb{Z}}
\newcommand{\bbQ}{\mathbb{Q}}
\newcommand{\bbR}{\mathbb{R}}

\begin{document}

\title{Linking numbers and triple linking numbers for periodic tangles}

\author[Y. Kotorii]{Yuka Kotorii}
\address[Kotorii]{Mathematics Program, Graduate School of Advanced Science and Engineering, Hiroshima University, 1-7-1 Kagamiyama, Higashi-Hiroshima, Hiroshima 739-8521, Japan}
\address[Kotorii]{International Institute for Sustainability with Knotted Chiral Meta Matter (WPI-SKCM$^2$), Hiroshima University, 1-3-1 Kagamiyama, Higashi-Hiroshima, Hiroshima 739-8531, Japan}
\address[Kotorii]{RIKEN Center for Interdisciplinary Theoretical and Mathematical Sciences (iTHEMS), 2-1 Hirosawa, Wako, Saitama 351-0198, Japan}
\email{kotorii@hiroshima-u.ac.jp}

\author[K. Yoshida]{Ken'ichi Yoshida}
\address[Yoshida]{International Institute for Sustainability with Knotted Chiral Meta Matter (WPI-SKCM$^2$), Hiroshima University, 1-3-1 Kagamiyama, Higashi-Hiroshima, Hiroshima 739-8531, Japan}
\email{kncysd@hiroshima-u.ac.jp}

\subjclass[2020]{57K10, 57M10}
\keywords{periodic tangles, linking numbers, Milnor link-homotopy invariants}
\date{}

\begin{abstract}
Periodic tangles are 1-dimensional submanifolds in the 3-space with translational symmetry.  
In this paper, we define the linking numbers for singly, doubly, and triply periodic tangles using appropriate motifs and show that they are well-defined and invariant under link-homotopy. 
Furthermore, we extend the notion to the triple linking numbers for singly, doubly, and triply periodic tangles, and the higher order linking numbers for singly periodic tangles under certain assumptions.
\end{abstract}

\maketitle

\section{Introduction}
\label{section:intro}

The linking number is one of the most fundamental invariants of a link in the 3-space. 
It can be defined in terms of various notions, such as the Gauss linking integral, the first homology group of the link complement, and the crossings in a diagram. 
The diagrammatic definition is easily generalized for the linking number of a link in a thickened surface (i.e., the product of a surface and the closed interval). 
Based on the homological definition, there is a generalization called the higher order linking number (also known as the Milnor invariant), introduced by Milnor~\cite{Milnor1}. 
The higher order linking number $\overline{\mu}_{L}(J)$ of an $m$-component link $L$ is defined for a sequence $J$ on $\{ 1, \dots, m \}$, using the lower central series of the fundamental group of the link complement. 
The linking numbers and specific Milnor invariants are invariant under link homotopy, which allows self-crossing changes. 

Periodic tangles are 1-dimensional submanifolds in the 3-space with translational symmetry. 
Depending on the number of translational directions, 
they are called singly, doubly, and triply periodic tangles. 
Periodic tangles have been used to model various entangled structures. 
For example, doubly periodic tangles represent textile structures \cite{GMO07}. 
A periodic tangle can be regarded as the preimage of a link in the solid torus $S^{1} \times D^{2}$, the thickened torus $T^{2} \times I$, or the 3-torus $T^{3}$ by the universal covering map. 
Then such a link is called a motif of a periodic tangle. 
Note that periodic lattices for translational symmetry of a periodic tangle are not fixed, 
and so there are many choices of motifs. 
Based on this fact, the equivalence of periodic tangles was introduced in \cite{DLM26}. 
In relation to this, compatibility between isotopies and finite covers of links in $S^{1} \times D^{2}$, $T^{2} \times I$, or $T^{3}$ was investigated in \cite{KMMY25}. 
As a result, a hyperbolic periodic tangle has a unique minimal motif \cite[Theorem 1.2]{KMMY25}. 

For periodic tangles, an analogue of the linking number, called the periodic linking number, was defined in \cite{Panagiotou15, PM18, PML13} 
using the Gauss linking integral with respect to a certain periodic lattice. 
While the periodic linking number of two closed components coincides with the ordinary linking number, 
that of two open components is not a topological invariant. 
Nonetheless, the periodic linking number has been applied to numerical data analysis of polymeric entangled structures. 
On the other hand, topological invariants are still useful for flexible structures. 
In \cite{Takano25}, the relationship between the topological and mechanical structures of knitted fabrics was investigated. 
In many cases, such topological structures are distinguished by the linking numbers of links in $T^{2} \times I$ as motifs. 

In this paper, we introduce the (higher order) linking numbers of components in a periodic tangle. 
The (higher order) linking numbers are invariant under equivalence of periodic tangles. 
More strongly, they are invariant under link-homotopy of periodic tangles, 
which is defined by allowing self-crossing changes. 
These invariants provide a convenient tool for distinguishing periodic tangles, 
which contributes to the topological analysis of flexible materials with periodic structures. 

In Section \ref{section:prelim}, 
we prepare notions for periodic tangles. 
In Section~\ref{section:good}, we introduce good and very good motifs of periodic tangles for fixed components, in which the images of the components are appropriately separated. 
They are appropriate for computing the (higher order) linking numbers.
In Section \ref{section:link}, 
we define the linking number $\lk(C_{1}, C_{2}) \in \frac{1}{2} \bbZ \cup \{ \pm \infty \}$ of two components $C_{1}$ and $C_{2}$ in a periodic tangle. 
For this purpose, we need to take a good motif. 
In a singly or doubly periodic tangle, $\lk(C_{1}, C_{2})$ is defined using the linking number of links in the solid torus or thickened torus. 
In a triply periodic tangle, $\lk(C_{1}, C_{2})$ is defined by regarding $C_{1}$ and $C_{2}$ as contained in a doubly periodic tangle. 
In Section \ref{section:higher}, we define the higher order linking number $\overline{\mu}_{C}(J)$ of components $C=(C_1,\dots,C_m)$ in a singly periodic tangle for a non-repeating sequence $J$ on $\{1,\dots,m\}$ under certain assumptions, using the Milnor link-homotopy invariants of a good motif. 
If all the components are open, it takes values in $\bbQ/\bbZ\cup\{\pm\infty\}$, and it is defined in the case that all the Milnor link-homotopy invariants of sufficiently small orders vanish; if some of them are closed, it takes values in $\bbZ/\Delta\bbZ$ for a suitable integer $\Delta$. 
In a doubly or triply periodic tangle, we define the triple linking number of three components that may contain open components with distinct directions, using a very good motif.
Furthermore, unless the components contain open components with distinct directions, we can define their higher order linking numbers by regarding them as contained in a singly periodic tangle. 

\section{Preliminaries}
\label{section:prelim}

A \emph{singly, doubly, or triply periodic tangle} is a 1-dimensional submanifold embedded in the interior of $\bbR \times D^{2}$, $\bbR^{2} \times I$, or $\bbR^{3}$ that is preserved by a translational action of $\bbZ$, $\bbZ^{2}$, or $\bbZ^{3}$, respectively. 
The quotient of a periodic tangle by a translational action is a link in the solid torus $S^{1} \times D^{2}$, the thickened torus $T^{2} \times I$, or the 3-torus $T^{3}$. 
Conversely, the preimage of a link in $S^{1} \times D^{2}$, $T^{2} \times I$, or $T^{3}$ by the universal covering map is a periodic tangle. 

Due to Diamantis, Lambropoulou, and Mahmoudi \cite{DLM26}, 
two periodic tangles are \emph{equivalent} 
if there is an ambient isotopy between them that is a composition of 
\begin{itemize}
\item periodic isotopies (i.e., isotopies equivariant under some translational actions) and 
\item affine transformations that are isotopic to the identity. 
\end{itemize}
Let $X$ be $S^{1} \times D^{2}$, $T^{2} \times I$, or $T^{3}$. 
A \emph{motif} of a periodic tangle $\tau$ is a link in $X$ whose preimage by the universal covering map is equivalent to $\tau$. 
Note that translational actions for periodic isotopies are not fixed. 
If the preimage $\widetilde{L} = p^{-1}(L)$ of a motif $L$ of a periodic tangle $\tau$ by a finite covering map $p \colon X \to X$, 
we call $\widetilde{L}$ a ``finite cover'' of $L$ for short. 
Then $\widetilde{L}$ is also a motif of $\tau$. 
In fact, two links $L_{0}$ and $L_{1}$ in $X$ are motifs of a common periodic tangle 
if and only if there is a common finite cover of $L_{0}$ and $L_{1}$ \cite[Proposition 2.4]{KMMY25}. 

For the universal covering map $\widetilde{X} \to X$, 
consider the deck transformation of $\pi_{1}(X)$ on $\widetilde{X}$. 
A fundamental domain for this action is a closed subset $F$ of $\widetilde{X}$ 
such that $\widetilde{X} = \bigcup_{g \in \pi_{1}(X)} g(F)$ and the interiors of $g(F)$ are disjoint. 
If $L$ is a link in $X$, 
we simply say that $F$ is a \emph{fundamental domain} for $L$. 
We will have to take a motif and a fundamental domain that are appropriate to compute the (higher order) linking numbers. 

A component of a periodic tangle is compact 
if and only if its image in a motif is null-homologous in $X$. 
Then we say that such a component is \emph{closed}. 
Otherwise we say that such a component is \emph{open}. 
Let $n = 1,2$, or $3$ so that $H_{1}(X; \bbZ) \cong \bbZ^{n}$. 
The \emph{direction} of an open component $C$ is the element $[v] \in \bbR P^{n-1} = (\bbR^{n} \setminus \{0\}) / \sim$ such that $C$ is preserved by the translation of $v$ in $\widetilde{X}$. 
Let $C_{1}$ and $C_{2}$ be two oriented open components of a periodic tangle $\tau$. 
Suppose that $C_{1}$ and $C_{2}$ are respectively projected onto components $K_{1}$ and $K_{2}$ in a motif of $\tau$. 
If two elements $[K_{1}], [K_{2}] \in H_{1}(X; \bbZ) \cong \bbZ^{n}$ have a non-zero common multiple, then $C_{1}$ and $C_{2}$ have a common direction. 
Otherwise $C_{1}$ and $C_{2}$ have distinct directions. 

The \emph{link-homotopy} is an equivalence relation among links that is generated by ambient isotopies and self-crossing changes, where a self-crossing change is a crossing change between arcs belonging to the same component.
The linking numbers and specific Milnor invariants for links in $S^{3}$ are invariant under link-homotopy. 
Similarly, we define a link-homotopy of periodic tangles as follows. 
Two periodic tangles are \emph{link-homotopic} if they are transformed by a composition of 
\begin{itemize}
\item periodic isotopies, 
\item affine transformations that are isotopic to the identity, and 
\item periodic self-crossing changes, 
\end{itemize}
where a periodic self-crossing change is equivariant under a translational action and performed between arcs belonging to the same component in a periodic tangle. 
A link-homotopy of periodic tangles induces a link-homotopy of links in the 3-manifold $X$. 
In general, however, the converse does not hold 
because two arcs for a self-crossing change in $X$ may have the preimages in distinct components by the universal covering map. 

\section{Good motifs of periodic tangles}
\label{section:good}

We need to take an appropriate motif of a periodic tangle to define the (higher order) linking numbers of components. 
One notion is a \emph{good} motif for fixed components, which has a fundamental domain that appropriately separates these components. 
We show that any motif has a good finite cover. 
The other stronger notion is a \emph{very good} motif for fixed components. 
We show that any motif has a very good finite cover after moving by a link-homotopy.

\subsection{Singly periodic tangle}

Let $C_{1}, \dots, C_{m}$ be oriented components of a singly periodic tangle $\tau$, 
and let $L$ be a motif of $\tau$. 
Suppose that $C_{1}, \dots, C_{m}$ are projected onto oriented components $K_{1}, \dots, K_{m}$ of $L$, respectively. 
We say that the motif $L$ is \emph{good} for $C_{1}, \dots, C_{m}$ 
if there is a fundamental domain $F$ for $L$ that satisfies the following conditions: 
\begin{itemize}
    \item $F = I \times D^{2}$ for a closed interval $I \subset \bbR$. 
    \item The interior of $F$ contains the closed components of $C_{1}, \dots, C_{m}$. 
    \item For each $i = 1, \dots, m$, if $C_{i}$ is an open component, then the homology class $[K_{i}] \in H_{1}(S^{1} \times D^{2}; \bbZ) \cong \bbZ$ is a generator. 
\end{itemize}
Then $F$ is called a \emph{good} fundamental domain for $L$. 
The components $K_{1}, \dots, K_{m}$ are necessarily distinct. 
Moreover, if $C_{i}$ is an open component, then $C_{i}$ is the preimage of $K_{i}$ by the universal covering map $\bbR \times D^{2} \to S^{1} \times D^{2}$. 

\begin{lem}
\label{lem:good_sp}
Let $C_{1}, \dots, C_{m}$ be oriented components of a singly periodic tangle $\tau$, 
and let $L$ be a motif of $\tau$. 
Then $L$ has a finite cover that is good for $C_{1}, \dots, C_{m}$. 
\end{lem}
\begin{proof}
Since the closed components of $C_{1}, \dots, C_{m}$ are compact, 
they are contained in a fundamental domain for a sufficiently large finite cover of $L$. 
If the homology class of $K_{i}$ is $n$ times a generator in $H_{1}(S^{1} \times D^{2}, \bbZ)$, 
then that of the preimage of $K_{i}$ by the $n$-sheeted covering map is a generator. 
If the homology class of $K_{i}$ is a generator in $H_{1}(S^{1} \times D^{2}; \bbZ)$, 
then that of the preimage of $K_{i}$ by any finite covering map is also a generator. 
Hence $L$ has a finite cover that is good for $C_{1}, \dots, C_{m}$. 
\end{proof}

A motif $L$ of $\tau$ is \emph{very good} for $C_{1}, \dots, C_{m}$ 
if $L$ is good, and each open component $C_{i}$ intersects the boundary of a good fundamental domain $F$ transversely at two points. 
Then $F$ is also called \emph{very good}. 

\begin{lem}
\label{lem:very_good_sp}
Let $C_{1}, \dots, C_{m}$ be oriented components of a singly periodic tangle $\tau$, 
and let $L$ be a motif of $\tau$. 
Then there is a link-homotopy from $\tau$ to a periodic tangle $\tau'$ with a motif $L'$ that satisfies the following conditions: 
\begin{itemize}
\item The link-homotopy from $\tau$ to $\tau'$ is projected to a link-homotopy from a finite cover of $L$ to $L'$. 
\item The motif $L'$ is very good for the corresponding components $C'_{1}, \dots, C'_{m}$ of $\tau'$. 
\end{itemize}
Moreover, if all $C_1,\dots,C_m$ are open, and $L$ is good for $C_1,\dots,C_m$, then we can take the finite cover of $L$ to be $L$ itself.
\end{lem}

We need the following lemma to show Lemma~\ref{lem:very_good_sp}. 
An open component $C$ of a singly periodic tangle is \emph{monotone} 
if the projection $\bbR \times D^{2} \to \bbR$ into the first coordinate is restricted to a homeomorphism  $C \to \bbR$. 
Then the component of a motif onto which $C$ is projected is also called monotone. 

\begin{lem}
\label{lem:monotone}
Let $\tau$ be a singly periodic tangle consisting of open oriented components $C_{1}, \dots, C_{m}$, 
and let $L$ be a good motif for $C_{1}, \dots, C_{m}$. 
Suppose that $C_{1}, \dots, C_{m}$ are projected onto oriented components $K_{1}, \dots, K_{m}$ of $L$, respectively. 
Suppose that the components $C_{1}, \dots, C_{m-1}$ are monotone. 
Then there is a link-homotopy from $L$ to a link $L'$ so that $K_{1}, \dots, K_{m-1}$ are fixed, 
and $K_{m}$ is homotoped to a monotone component $K'_{m}$. 
Moreover, this link-homotopy from $L$ is lifted to a link-homotopy from $\tau$. 
\end{lem}
\begin{proof}
The space $M = S^{1} \times D^{2} \setminus K_{1} \cup \dots \cup K_{m-1}$ is homeomorphic to the mapping torus of a homeomorphism on the $m-1$-punctured disk, 
and the fundamental group of $M$ is generated by the classes of loops $a_{1}, \dots, a_{m-1}$ in the fiber $\Sigma = \{ 0 \} \times D^{2} \setminus K_{1} \cup \dots \cup K_{m-1}$ and a monotone loop $b$ in $S^{1} \times D^{2}$. 
Hence, $K_{m}$ is homotopic in $M$ to a concatenation of a sequence of loops $a_{1}, \dots, a_{m-1}, b$, and their inverses. 
Since $L$ is good, $K_{m}$ is homotopic to the concatenation of a loop $a'$ in the fiber $\Sigma$ 
and a single loop $b$ or $b^{-1}$. 
By modifying the loop $a'$ around the fiber $\Sigma$, 
$K_{m}$ is homotopic to a monotone loop. 
Thus, we obtain the desired link-homotopy from $L$. 
Since the preimage of each component of the good motif $L$ is a single component, 
any link-homotopy from $L$ lifts to a link-homotopy from $\tau$. 
\end{proof}

\begin{proof}[Proof of Lemma~\ref{lem:very_good_sp}]
By Lemma~\ref{lem:good_sp}, we may assume that $L$ is good. 
Moreover, we may assume that $C_{1}, \dots, C_{m_{0}}$ are open, and $C_{m_{0}+1}, \dots, C_{m}$ are closed. 
Inductively by Lemma~\ref{lem:monotone}, $\tau$ is link-homotopic to $\tau'$ so that $C_{1}, \dots, C_{m_{0}}$ are homotoped to monotone components $C'_{1}, \dots, C'_{m_{0}}$. 
Then the closed components $C_{m_{0}+1}, \dots, C_{m}$ are homotoped to $C'_{m_{0}+1}, \dots, C'_{m}$. 
Let $L^{(0)}$ be the motif obtained from $L$ by the projected link-homotopy. 
Note that $L^{(0)}$ is not necessarily good if there are closed components. 
By Lemma~\ref{lem:good_sp}, there is a finite cover $L'$ of $L^{(0)}$ that is good for $C'_{1}, \dots, C'_{m}$. 
Since $C'_{i}$ for each $1 \leq i \leq m_{0}$ is monotone, 
$C'_{i}$ intersects the boundary of a good fundamental domain transversely at two points. 
Hence the motif $L'$ is very good. 
\end{proof}

\subsection{Doubly periodic tangle}

Similarly to the case for a singly periodic tangle, 
let $C_{1}, \dots, C_{m}$ be oriented components of a doubly periodic tangle $\tau$, 
and let $L$ be a motif of $\tau$. 
Suppose that $C_{1}, \dots, C_{m}$ are projected onto oriented components $K_{1}, \dots, K_{m}$ of $L$, respectively. 
We say that the motif $L$ is \emph{good} for $C_{1}, \dots, C_{m}$ 
if there is a fundamental domain $F$ for $L$ that satisfies the following conditions: 
\begin{itemize}
\item $F = P \times I$ for a parallelogram $P \subset \bbR^{2}$. 
\item The opposite pairs of edges of the parallelogram $P$ are glued via translations by elements that form a basis of $\pi_{1} (T^{2} \times I) \cong \bbZ^{2}$. 
\item The interior of $F$ contains the closed components of $C_{1}, \dots, C_{m}$. 
\item For each $i = 1, \dots, m$, if $C_{i}$ is an open component, then 
\begin{itemize}
\item the homology class $[K_{i}] \in H_{1}(T^{2} \times I; \bbZ) \cong \bbZ^{2}$ is a primitive element, and 
\item the interior of the union $\bigcup_{g \in G_{i}} g(F)$ contains $C_{i}$, 
where $G_{i}$ is the subgroup of $\pi_{1} (T^{2} \times I) \cong H_{1}(T^{2} \times I; \bbZ) \cong \bbZ^{2}$ generated by $[K_{i}]$. 
\end{itemize}
\end{itemize}
Then $F$ is called a \emph{good} fundamental domain for $L$. 
The components $K_{1}, \dots, K_{m}$ are necessarily distinct. 
Moreover, if $C_{i}$ is an open component, then an opposite pair of edges of $P$ is glued via the translation by $[K_{i}]$. 
There are at most two distinct directions of the open components of $C_{1}, \dots, C_{m}$. 
If $C_{i}$ and $C_{j}$ are open components with distinct directions, 
then the homology classes $[K_{i}], [K_{j}] \in H_{1}(T^{2} \times I; \bbZ) \cong \bbZ^{2}$ form a basis. 

\begin{lem}
\label{lem:good_dp}
Let $C_{1}, \dots, C_{m}$ be oriented components of a doubly periodic tangle $\tau$, 
and let $L$ be a motif of $\tau$. 
Suppose that there are at most two distinct directions of open components of $C_{1}, \dots, C_{m}$. 
Then $L$ has a finite cover that is good for $C_{1}, \dots, C_{m}$. 
\end{lem}
\begin{proof}
We argue in the same manner as in the proof of Lemma~\ref{lem:good_sp}. 
Since the closed components of $C_{1}, \dots, C_{m}$ are compact, 
they are contained in a fundamental domain for a sufficiently large finite cover of $L$. 
If the homology class of $K_{i}$ is $n$ times a primitive element in $H_{1}(T^{2} \times I; \bbZ)$, 
then that of the preimage of $K_{i}$ by an $n$-sheeted covering map is a primitive element. 
If the homology class of $K_{i}$ is a primitive element in $H_{1}(T^{2} \times I; \bbZ)$, 
then that of the preimage of $K_{i}$ by any finite covering map is also a primitive element. 
Since there are at most two distinct directions of open components of $C_{1}, \dots, C_{m}$, 
we may assume that $[K_{1}], \dots, [K_{m}]$ are contained in a basis of $H_{1}(T^{2} \times I; \bbZ)$ by taking a finite cover of $L$. 
Since an open component $C_{i}$ has a bounded range in the direction orthogonal to that of $C_{i}$, 
we can take a fundamental domain $F$ for a sufficiently large finite cover of $L$ 
so that an open component $C_{i}$ is contained in $\bigcup_{g \in G_{i}} g(F)$. 
Hence $L$ has a finite cover that is good for $C_{1}, \dots, C_{m}$. 
\end{proof}

Note that if $m = 2$, then a motif $L$ always has a finite cover that is good for $C_{1}$ and $C_{2}$. 

A motif $L$ of $\tau$ is \emph{very good} for $C_{1}, \dots, C_{m}$ 
if $L$ is good, and each open component $C_{i}$ intersects the boundary of a good fundamental domain $F$ transversely at two points. 
Then $F$ is also called \emph{very good}. 

\begin{lem}
\label{lem:very_good_dp}
Let $C_{1}, \dots, C_{m}$ be oriented components of a doubly periodic tangle $\tau$, 
and let $L$ be a motif of $\tau$. 
Suppose that there are at most two distinct directions of the open components of $C_{1}, \dots, C_{m}$. 
Then there is a link-homotopy from $\tau$ to a periodic tangle $\tau'$ with a motif $L'$ that satisfies the following conditions: 
\begin{itemize}
\item The link-homotopy from $\tau$ to $\tau'$ is projected to a link-homotopy from a finite cover of $L$ to $L'$. 
\item The motif $L'$ is very good for the corresponding components $C'_{1}, \dots, C'_{m}$ of $\tau'$. 
\end{itemize}
\end{lem}
\begin{proof}
By Lemma~\ref{lem:good_dp}, we may assume that $L$ is good. 
We may assume that 
\begin{itemize}
\item $C_{1}, \dots, C_{m_{0}+m_{1}}$ are open, 
\item $C_{m_{0}+m_{1}+1}, \dots, C_{m}$ are closed, 
\item $[K_{1}] = \dots = [K_{m_{0}}] \neq [K_{m_{0}+1}] = \dots = [K_{m_{0}+m_{1}}] \in H_{1}(T^{2} \times I; \bbZ)$, and 
\item the directions of $C_{1}$ and $C_{m_{0}+1}$ are orthogonal if $m_{0}, m_{1} \geq 1$. 
\end{itemize}
We say that $C_{i}$ for $1 \leq i \leq m_{0} + m_{1}$ is monotone 
if the orthogonal projection $\bbR^{2} \times I \to \bbR$ onto the direction of $C_{i}$ is restricted to a homeomorphism $C_{i} \to \bbR$. 
In the same manner as in the proof of Lemma~\ref{lem:very_good_sp}, 
by a link-homotopy, we may assume that $C_{1}, \dots, C_{m_{0}}$ are monotone. 
Moreover, by taking a finite cover, we may assume that $L$ is a good motif with a good fundamental domain $F$ for $C_{1}, \dots, C_{m}$. 

Let $n > 0$ such that a connected subspace in $C_{m_{0}+1}$ containing $C_{m_{0}+1} \cap F$ is contained in the interior of $F_{1} = \bigcup_{-n \leq k \leq n} g_{1}^{k}(F)$, 
where $g_{1} \in \pi_{1}(T^{2} \times I)$ is the element corresponding to $[K_{m_{0}+1}]$. 
Then there is a link-homotopy between $\tau$ and $\tau_{1}$ that satisfies the following: 
\begin{itemize}
    \item $F_{1}$ is preserved as a fundamental domain. 
    \item $C_{1}, \dots, C_{m_{0}}$, and $C_{m_{0}+1} \cap F$ are fixed. 
    \item $C_{m_{0}+1}$ is homotoped to $C'_{m_{0}+1}$. 
    \item $C'_{m_{0}+1}$ intersects $\partial F_{1}$ at two points. 
\end{itemize}
Inductively by link-homotopies fixing $C_{1} = C'_{1}, \dots, C_{m_{0}} = C'_{m_{0}}, C'_{m_{0}+1}, \dots, C'_{m_{0}+i}$, 
we obtain a doubly periodic tangle $\tau'$ with a fundamental domain $F_{m_{1}}$ 
such that $C'_{i}$ for each $1 \leq i \leq m_{0} + m_{1}$ intersects $F_{m_{1}}$ at two points. 
Then the closed components $C_{m_{0}+m_{1}+1}, \dots, C_{m}$ are homotoped to some components $C'_{m_{0}+m_{1}+1}, \dots, C'_{m}$. 
By taking a finite cover, we obtain a good motif $L'$ with a good fundamental domain $F'$ that is a union of copies of $F_{m_{1}}$. 
Then the motif $L'$ is very good for $C'_{1}, \dots, C'_{m}$. 
\end{proof}

\subsection{Triply periodic tangle}

Similarly to the cases for singly and doubly periodic tangles, 
let $C_{1}, \dots, C_{m}$ be oriented components of a triply periodic tangle $\tau$, 
and let $L$ be a motif of $\tau$. 
Suppose that $C_{1}, \dots, C_{m}$ are projected onto oriented components $K_{1}, \dots, K_{m}$ of $L$, respectively. 
We say that the motif $L$ is \emph{good} for $C_{1}, \dots, C_{m}$ 
if there is a fundamental domain $F$ for $L$ that satisfies the following conditions: 
\begin{itemize}
\item A parallelepiped $F \subset \bbR^{3}$ is a fundamental domain for $L$. 
\item The opposite pairs of faces of $F$ are glued via translations by elements that form a basis of $\pi_{1} (T^{3}) \cong \bbZ^{3}$. 
\item The interior of $F$ contains the closed components of $C_{1}, \dots, C_{m}$. 
\item For each $i = 1, \dots, m$, if $C_{i}$ is an open component, then 
\begin{itemize}
\item the homology class $[K_{i}] \in H_{1}(T^{3}; \bbZ) \cong \bbZ^{3}$ is a primitive element, and 
\item the interior of the union $\bigcup_{g \in G_{i}} g(F)$ contains $C_{i}$, 
where $G_{i}$ is the subgroup of $\pi_{1} (T^{3}) \cong H_{1}(T^{3}; \bbZ) \cong \bbZ^{3}$ generated by $[K_{i}]$. 
\end{itemize}
\end{itemize}
Then $F$ is called a \emph{good} fundamental domain for $L$. 
The components $K_{1}, \dots, K_{m}$ are necessarily distinct. 
Moreover, if $C_{i}$ is an open component, then an opposite pair of faces of $F$ is glued via the translation by $[K_{i}]$. 
There are at most three distinct directions of open components of $C_{1}, \dots, C_{m}$. 
If there are three distinct directions of open components of $C_{1}, \dots, C_{m}$, 
they are linearly independent. 
If $C_{i}$, $C_{j}$, and $C_{k}$ are open components with distinct directions, 
then the homology classes $[K_{i}], [K_{j}], [K_{k}] \in H_{1}(T^{3}; \bbZ) \cong \bbZ^{3}$ form a basis. 

A motif $L$ of $\tau$ is \emph{very good} for $C_{1}, \dots, C_{m}$ 
if $L$ is good, and each open component $C_{i}$ intersects the boundary of a good fundamental domain $F$ transversely at two points. 
Then $F$ is also called \emph{very good}. 

The following lemma is proved in the same manner as Lemmas~\ref{lem:good_dp} and \ref{lem:very_good_dp}. 

\begin{lem}
\label{lem:good_tp}
Let $C_{1}, \dots, C_{m}$ be oriented components of a triply periodic tangle $\tau$, 
and let $L$ be a motif of $\tau$. 
Suppose that there are at most three distinct directions of the open components of $C_{1}, \dots, C_{m}$, 
and if there are three distinct directions of them, they are linearly independent.  
Then $L$ has a finite cover that is good for $C_{1}, \dots, C_{m}$. 
Moreover, there is a link-homotopy from $\tau$ to a periodic tangle $\tau'$ with a motif $L'$ that satisfies the following conditions: 
\begin{itemize}
\item The link-homotopy from $\tau$ to $\tau'$ is projected to a link-homotopy from a finite cover of $L$ to $L'$. 
\item The motif $L'$ is very good for the corresponding components $C'_{1}, \dots, C'_{m}$ of $\tau'$. 
\end{itemize}
\end{lem}

\section{Linking numbers}
\label{section:link}

\subsection{Singly periodic tangle}

Let $K$ be an oriented link in the solid torus $S^{1} \times D^{2}$ that contains two components $K_{1}$ and $K_{2}$. 
We obtain a diagram of $K$ on $S^{1} \times I$ in general position by the standard projection. 
The \emph{linking number} of $K_{1}$ and $K_{2}$, denoted by $\lk(K_{1}, K_{2})$, is defined to be half of the sum of the signs of crossings of $K_{1}$ with $K_{2}$. 
Since the Reidemeister moves and self-crossing changes do not change $\lk(K_{1}, K_{2})$, it is invariant under link-homotopy. 
Remark that $\lk (K_{1}, K_{2})$ coincides with the linking number of $\iota (K_{1})$ and $\iota (K_{2})$ in $\bbR^{3}$, 
where $\iota \colon S^{1} \times D^{2} \to \bbR^{3}$ is a standard embedding 
such that the linking number of the images of the longitude and the core is equal to zero. 

Let $C_{1}$ and $C_{2}$ be two oriented components of a singly periodic tangle $\tau$. 
The \emph{linking number} of $C_{1}$ and $C_{2}$, denoted by $\lk(C_{1}, C_{2})$, is defined as follows. 
By Lemma~\ref{lem:good_sp}, there is a good motif $L$ of $\tau$ for $C_{1}$ and $C_{2}$. 

\begin{dfn}
Let $L$ be a good motif of a singly periodic tangle $\tau$ for oriented components $C_{1}$ and $C_{2}$. 
Suppose that $C_{1}$ and $C_{2}$ are respectively projected onto components $K_{1}$ and $K_{2}$ of $L$. 
If both $C_{1}$ and $C_{2}$ are open, 
we define $\lk(C_{1}, C_{2}) \in \{ 0, \pm \infty \}$ 
so that the sign coincides with that of $\lk(K_{1}, K_{2})$. 
Otherwise we define $\lk(C_{1}, C_{2}) = \lk(K_{1}, K_{2}) \in \bbZ$. 
\end{dfn}

\begin{ex}
The left of Figure~\ref{fig:singly_exa1} indicates motifs of singly periodic tangles, 
which are not good for any pair of components. 
The right of Figure~\ref{fig:singly_exa1} indicates good motifs for components $C_{1}$ and $C_{2}$. 
The right motifs are finite covers of the left motifs. 
The linking number $\lk(C_{1}, C_{2})$ of the first example is $+\infty$, and those of the second and third are 1. 
Note that the linking numbers cannot be well computed using the left motifs. 
    \begin{figure}[h]
  \centering
\includegraphics[angle=90, width=0.8\textwidth]{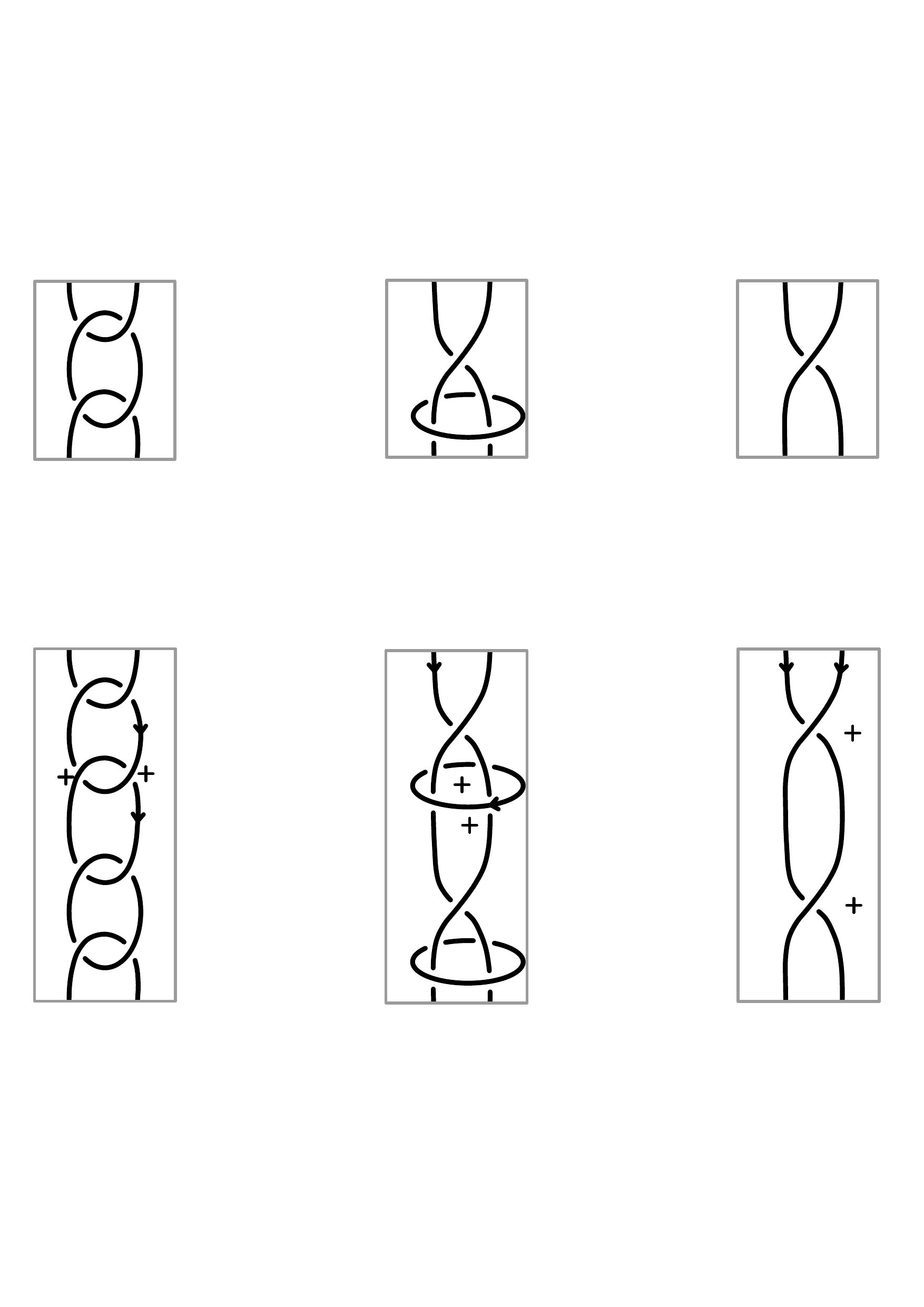} 
 \put(-160,168){$C_1$}\put(-160,184){$C_2$}\put(-160,93){$C_1$}\put(-118,122){$C_2$}\put(-130,43){$C_1$}\put(-110,43){$C_2$}
  \caption{Non-good and good motifs of singly periodic tangles }
  \label{fig:singly_exa1} 
\end{figure}
\end{ex}

\begin{thm}
\label{thm:lk_sp}
Let $C_{1}$ and $C_{2}$ be two oriented components of a singly periodic tangle $\tau$. 
Then the linking number $\lk(C_{1}, C_{2})$ does not depend on the choice of a good motif. 
Moreover, if there is a link-homotopy from $\tau$ to a singly periodic tangle $\tau'$ which maps $C_{1}$ and $C_{2}$ respectively to components $C'_{1}$ and $C'_{2}$, 
then $\lk(C_{1}, C_{2}) = \lk(C'_{1}, C'_{2})$. 
In other words, the linking number is well-defined for a link-homotopy class of a singly periodic tangle. 
\end{thm}
\begin{proof}
Note that any finite cover of a good motif is also good by the proof of Lemma~\ref{lem:good_sp}. 
The second assertion follows from the first assertion and the fact that 
the linking number of a link in $S^{1} \times D^{2}$ is invariant under link-homotopy. 

To show the first assertion, 
let $L$ be a good motif of $\tau$ for $C_{1}$ and $C_{2}$, 
and let $L'$ be the $n$-sheeted cover of $L$ for $n>1$. 
Suppose that $K_{i}$ and $K'_{i}$ for $i =1,2$ are the images of $C_{i}$ respectively in $L$ and $L'$. 
Let $F$ be a good fundamental domain for $L$. 
If $C_{i}$ is open, then $K'_{i}$ is the preimage of $K_{i}$ by the finite covering map. 
If $C_{i}$ is closed, then $K'_{i}$ is a component of the preimage of $K_{i}$ by the finite covering map. 
If both $C_{1}$ and $C_{2}$ are open, then $\lk(K'_{1}, K'_{2}) = n \lk(K_{1}, K_{2})$. 
Otherwise the crossings of $K'_{1}$ and $K'_{2}$ are contained in a single unit that is the image of $F$ by the universal covering map $\bbR \times D^{2} \to S^{1} \times D^{2}$ for $L'$, 
and so $\lk(K'_{1}, K'_{2}) = \lk(K_{1}, K_{2})$. 
Hence the linking number $\lk(C_{1}, C_{2})$ is well-defined. 
\end{proof}

Consider a diagram of the singly periodic tangle $\tau$ on $\bbR \times I$. 
If $C_{1}$ or $C_{2}$ is closed, they have only finitely many crossings. 
In this case, the linking number $\lk(C_{1}, C_{2})$ can be computed by the ordinary way using a diagram on $\bbR \times I$. 
A good fundamental domain $F$ is taken for a technical reason to count the crossings. 
If $C_{1}$ and $C_{2}$ are open, they may have infinitely many crossings. 
In the case that all the infinitely many crossings have the same sign, 
the linking number $\lk(C_{1}, C_{2})$ is equal to $+\infty$ or $-\infty$ depending on the sign. 
In the case that there are an infinite number of both positive and negative crossings, 
we have to count the crossings by taking a quotient to a solid torus. 
For example, the linking number $\lk(C_{1}, C_{2})$ for the singly periodic tangle shown in Figure~\ref{fig:singly_exa2} is equal to $+\infty$. 

    \begin{figure}[h]
  \centering
\includegraphics[ width=0.3\textwidth]{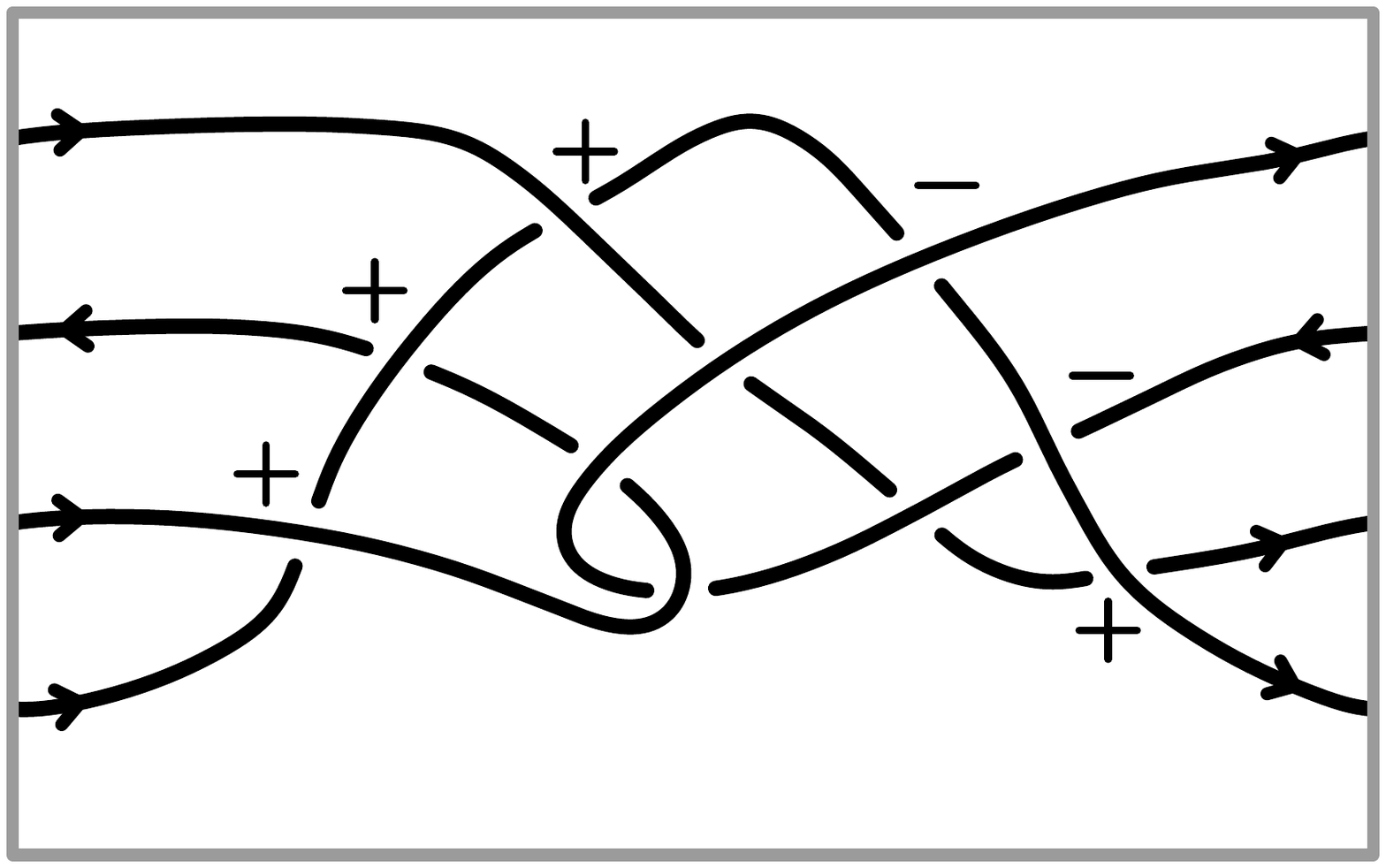} 
 \put(-125,59){$C_2$}\put(-125,15){$C_1$}
  \caption{A diagram calculation of the linking number for a singly periodic tangle}
  \label{fig:singly_exa2} 
\end{figure}

\subsection{Doubly periodic tangle}\label{subsec:Dp-tangle}

Let $K$ be an oriented link in the thickened torus $T^{2} \times I$ that contains two components $K_{1}$ and $K_{2}$. 
We obtain a diagram of $K$ on $T^{2}$ in general position by the standard projection. 
The \emph{linking number} of $K_{1}$ and $K_{2}$, denoted by $\lk(K_{1}, K_{2})$, is defined to be half of the sum of the signs of crossings of $K_{1}$ with $K_{2}$. 
Since the Reidemeister moves and self-crossing changes do not change $\lk(K_{1}, K_{2})$, it is invariant under link-homotopy. 
Note that $\lk(K_{1}, K_{2}) \in \frac{1}{2}\bbZ$ is not necessarily an integer 
because the number of crossings may be odd. 

Let $C_{1}$ and $C_{2}$ be two oriented components of a doubly periodic tangle $\tau$.  
Similarly to the case for a singly periodic tangle, 
the \emph{linking number} of $C_{1}$ and $C_{2}$, denoted by $\lk(C_{1}, C_{2})$, is defined as follows. 
By Lemma~\ref{lem:good_dp}, there is a good motif $L$ of $\tau$ for $C_{1}$ and $C_{2}$. 

\begin{dfn}
Let $L$ be a good motif of a doubly periodic tangle $\tau$ for oriented components $C_{1}$ and $C_{2}$. 
Suppose that $C_{1}$ and $C_{2}$ are respectively projected onto components $K_{1}$ and $K_{2}$ of $L$. 
If $C_{1}$ and $C_{2}$ are open components with a common direction, 
we define $\lk(C_{1}, C_{2}) \in \{ 0, \pm \infty \}$ 
so that the sign coincides with that of $\lk(K_{1}, K_{2})$. 
Otherwise, we define $\lk(C_{1}, C_{2}) = \lk(K_{1}, K_{2}) \in \frac{1}{2}\bbZ$. 
\end{dfn}

\begin{ex}
The left of Figure~\ref{fig:doubly_exa1} indicates motifs of doubly periodic tangles, 
which are not good for any pair of components. 
The right of Figure~\ref{fig:doubly_exa1} indicates good motifs for components $C_{1}$ and $C_{2}$. 
The right motifs are finite covers of the left motifs. 
The linking number $\lk(C_{1}, C_{2})$ of the first example is $+ \infty$, and that of the second is $1/2$. 
Note that the linking numbers cannot be well computed using the left motifs. 
\end{ex}

    \begin{figure}[h]
  \centering
\includegraphics[ width=0.6\textwidth]{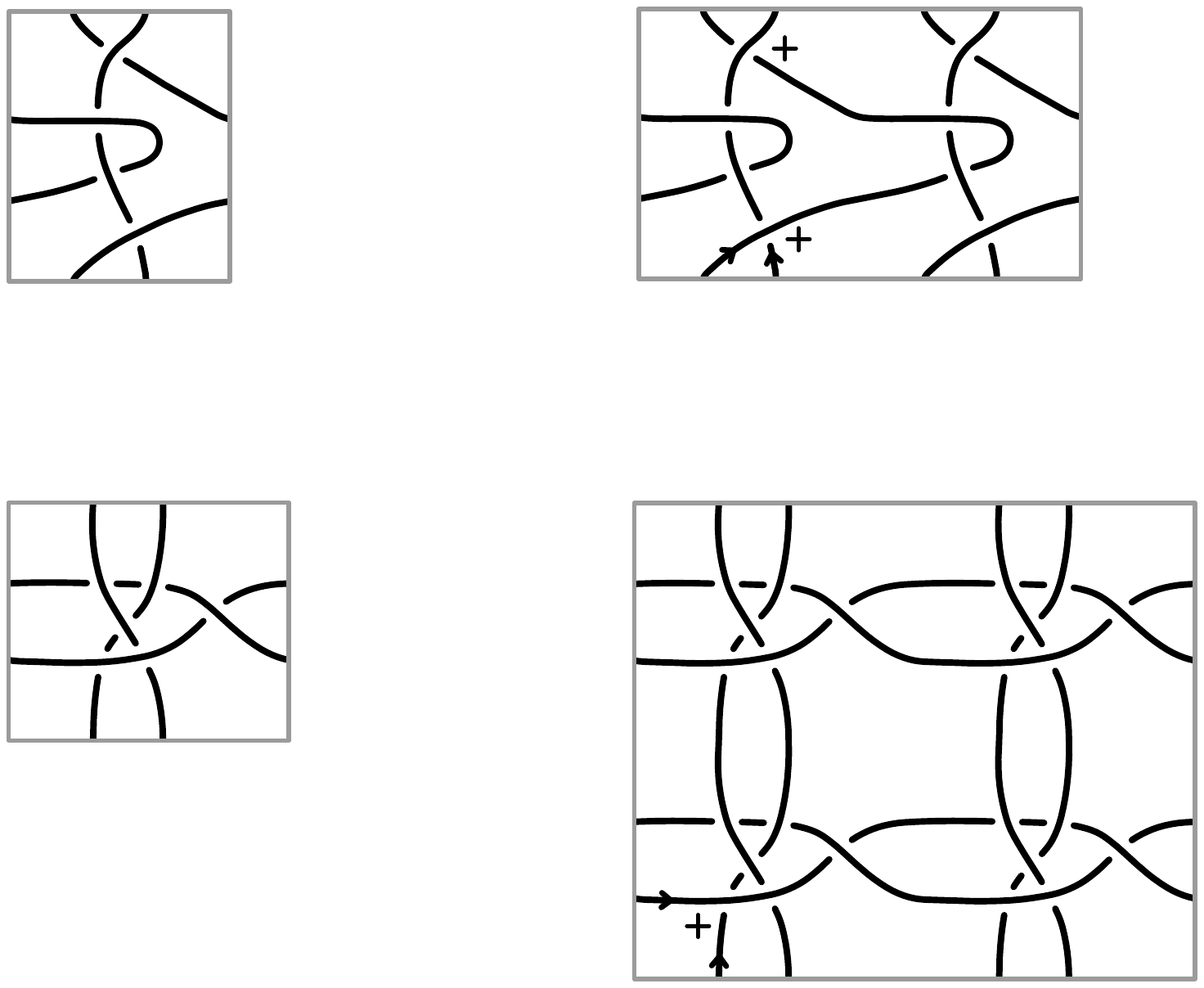} 
 \put(-108,97){$C_1$}\put(-94,97){$C_2$}\put(-128,12){$C_1$}\put(-105,-7){$C_2$}
  \caption{Non-good and good motifs of doubly periodic tangles}
  \label{fig:doubly_exa1} 
\end{figure}

\begin{thm}
\label{thm:lk_dp}
Let $C_{1}$ and $C_{2}$ be two oriented components of a doubly periodic tangle $\tau$. 
Then the linking number $\lk(C_{1}, C_{2})$ does not depend on the choice of a good motif. 
Moreover, if there is a link-homotopy from $\tau$ to a doubly periodic tangle $\tau'$ which maps $C_{1}$ and $C_{2}$ respectively to components $C'_{1}$ and $C'_{2}$, 
then $\lk(C_{1}, C_{2}) = \lk(C'_{1}, C'_{2})$. 
In other words, the linking number is well-defined for a link-homotopy class of a doubly periodic tangle. 
\end{thm}
\begin{proof}
Note that a finite cover of a good motif is not necessarily good. 
For instance, 
if $C_{1}$ and $C_{2}$ are open components with distinct directions, 
the homology classes $[K_{1}], [K_{2}] \in H_{1}(T^{2} \times I; \bbZ) \cong \bbZ^{2}$ form a basis for a good motif. 
Hence a good finite cover of this motif has a basis consisting of $m[K_{1}]$ and $n[K_{2}]$ for $m, n \geq 1$. 
Nonetheless, Lemma~\ref{lem:good_dp} implies that two good motifs of $\tau$ have a common good finite cover. 
The second assertion follows from the first assertion and the fact that 
the linking number of a link in $T^{2} \times I$ is invariant under link-homotopy. 

To show the first assertion, let $L'$ be a good finite cover of a good motif $L$ of $\tau$. 
Suppose that $K_{i}$ and $K'_{i}$ for $i =1,2$ are the images of $C_{i}$ respectively in $L$ and $L'$. 
Let $F$ be a good fundamental domain for $L$. 
The knot $K'_{i}$ is a component of the preimage of $K_{i}$ by the finite covering map. 
If $C_{i}$ is open, then the image of $\bigcup_{g \in G_{i}} g(F)$ contains $K'_{i}$, and is disjoint from the other components of the preimage of $K_{i}$, where $G_{i}$ is the subgroup of $\pi_{1} (T^{2} \times I) \cong H_{1}(T^{2} \times I; \bbZ) \cong \bbZ^{2}$ generated by $[K_{i}]$. 
If $C_{1}$ and $C_{2}$ are open components with a common direction, then $\lk(K'_{1}, K'_{2})$ is a positive multiple of $\lk(K_{1}, K_{2})$. 
Otherwise the crossings of $K'_{1}$ and $K'_{2}$ are contained in a single unit that is the image of $F$ by the universal covering map $\bbR^{2} \times I \to T^{2} \times I$ for $L'$, 
and so $\lk(K'_{1}, K'_{2}) = \lk(K_{1}, K_{2})$. 
Hence the linking number $\lk(C_{1}, C_{2})$ is well-defined. 
\end{proof}

If $C_{1}$ and $C_{2}$ are open components with distinct directions, 
the linking number $\lk(C_{1}, C_{2})$ is the sum of an integer and $1/2$. 
Otherwise $C_{1}$ and $C_{2}$ are contained in a singly periodic tangle. 
For example, Figures~\ref{fig:doubly_exa2}, \ref{fig:doubly_exa3} indicate doubly periodic tangles and components $C_{1}$ and $C_{2}$. 
In Figure~\ref{fig:doubly_exa2}, $C_{1}$ and $C_{2}$ form a singly periodic tangle. 
In Figure~\ref{fig:doubly_exa3}, $C_{1}$ and translational images of $C_{2}$ form a singly periodic tangle. 

    \begin{figure}[h]
  \centering
\includegraphics[ width=0.4\textwidth]{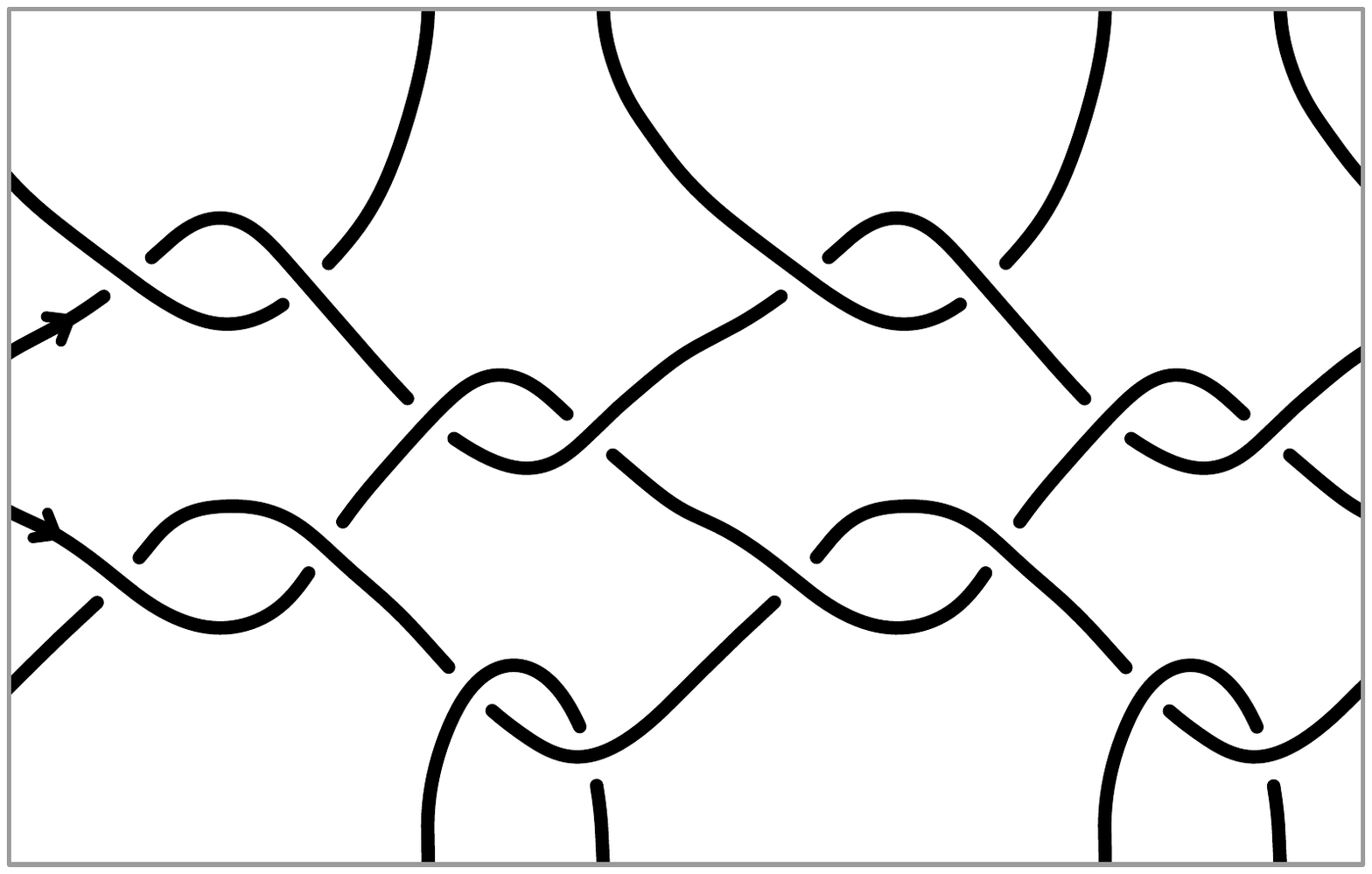} 
 \put(-155,57){$C_2$}\put(-155,40){$C_1$}
  \caption{$C_{1}$ and $C_{2}$ form a singly periodic tangle.}
  \label{fig:doubly_exa2} 
\end{figure}
    \begin{figure}[h]
  \centering
\includegraphics[ width=0.4\textwidth]{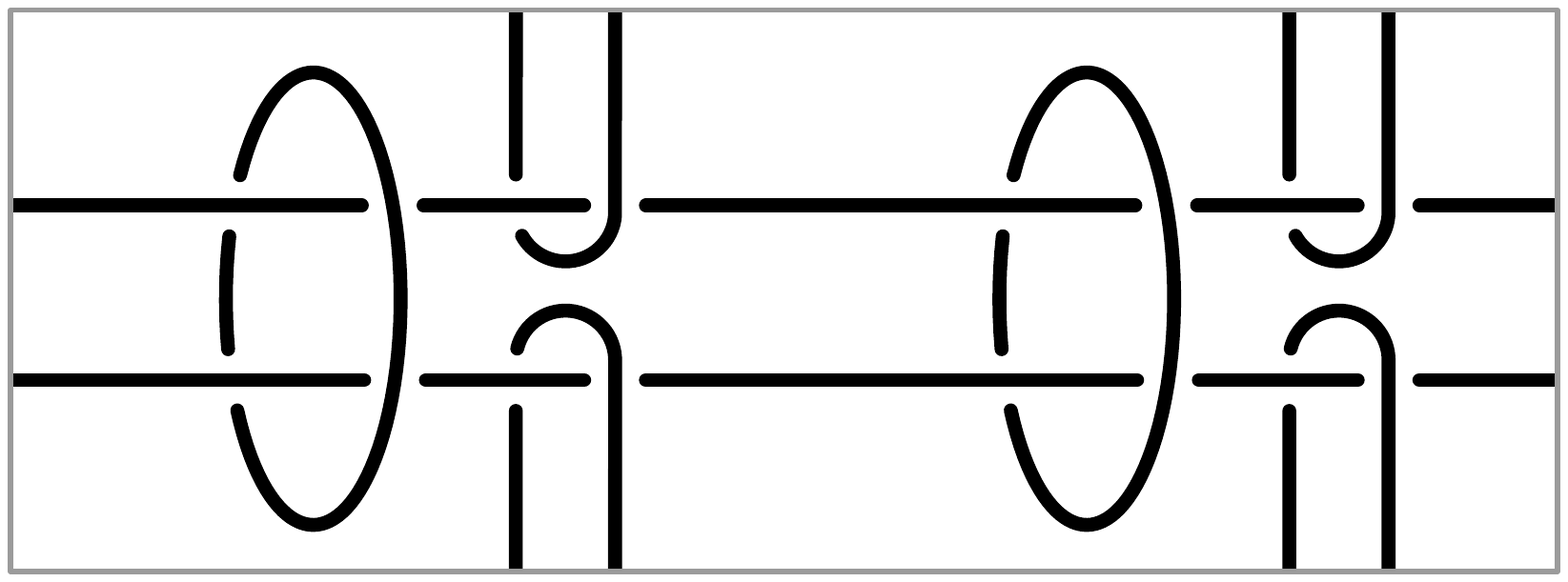} 
 \put(-155,47){$C_1$}\put(-118,20){$C_2$}
  \caption{$C_{1}$ and translational images of $C_{2}$ form a singly periodic tangle.}
  \label{fig:doubly_exa3} 
\end{figure}

\begin{prop}
\label{prop:sp_in_dp}
Let $C_{1}$ and $C_{2}$ be two oriented components of a doubly periodic tangle $\tau$. 
Unless $C_{1}$ and $C_{2}$ are open components with distinct directions, 
there are 
\begin{itemize}
\item a singly periodic tangle $\tau_{0} \subset \bbR \times D^{2}$ 
preserved by a translational action $\rho \colon \bbZ \to \Isom(\bbR \times D^{2})$, 
\item an orientation-preserving embedding $\iota \colon \bbR \times D^{2} \to \bbR^{2} \times I$, and 
\item components $C_{1,0}$ and $C_{2,0}$ of $\tau_{0}$ 
\end{itemize}
such that 
\begin{itemize}
\item $\rho$ is a restriction of a translational action for $\tau$, 
\item $\iota (\tau_{0}) \subset \tau$, and 
\item $\iota (C_{i,0}) = C_{i}$ for $i = 1, 2$. 
\end{itemize}
Then we shortly say that $\tau$ contains $\tau_{0}$, and $\tau_{0}$ contains $C_{1}$ and $C_{2}$. 
Moreover, the linking number $\lk(C_{1}, C_{2})$ in $\tau$ coincides with $\lk(C_{1,0}, C_{2,0})$ in $\tau_{0}$. 
\end{prop}
\begin{proof}
Take a good motif $L$ of $\tau$ for $C_{1}$ and $C_{2}$. 
Let $F$ be a good fundamental domain for $L$. 
Let $G$ be a subgroup of $\pi_{1} (T^{2} \times I) \cong \bbZ^{2}$ 
generated by an element that translates a face of $F$ to the opposite face. 
Unless $C_{1}$ and $C_{2}$ are open components with distinct directions, 
we may assume that $G$ contains the classes $[K_{1}]$ and $[K_{2}]$. 
If $C_{1}$ and $C_{2}$ are open components with a common direction, 
they are preserved by the translational action of $G$. 
Hence $C_{1} \cup C_{2}$ is the image of a singly periodic tangle $\tau_{0}$ by an embedding $\iota \colon \bbR \times D^{2} \to \bbR^{2} \times I$ with $\iota (\bbR \times D^{2}) \subset \bigcup_{g \in G} g(F)$. 
If $C_{1}$ or $C_{2}$ is closed, 
$\bigcup_{g \in G} g(C_{1}) \cup g(C_{2})$ is the image of a singly periodic tangle $\tau_{0}$ by an embedding $\iota$. 
The coincidence of linking numbers follows from the fact that the projections to diagrams are compatible. 
\end{proof}

\subsection{Triply periodic tangle}

The linking number for a link in $T^{3}$ is not defined at present. 
We define the linking number of two components in a triply periodic tangle 
using a doubly periodic tangle that contains these two components. 
The following proposition holds similarly to Proposition~\ref{prop:sp_in_dp}. 

\begin{prop}
\label{prop:dp_in_tp}
Let $C_{1}$ and $C_{2}$ be components of a triply periodic tangle $\tau$. 
Then there are 
\begin{itemize}
\item a doubly periodic tangle $\tau_{0} \subset \bbR^{2} \times I$ 
preserved by a translational action $\rho \colon \bbZ^{2} \to \Isom(\bbR^{2} \times I)$, 
\item an orientation-preserving embedding $\iota \colon \bbR^{2} \times I \to \bbR^{3}$, and 
\item components $C_{1,0}$ and $C_{2,0}$ of $\tau_{0}$ 
\end{itemize}
such that 
\begin{itemize}
\item $\rho$ is a restriction of a translational action for $\tau$, 
\item $\iota (\tau_{0}) \subset \tau$, and 
\item $\iota (C_{i,0}) = C_{i}$ for $i = 1, 2$. 
\end{itemize}
Then we shortly say that $\tau$ contains $\tau_{0}$, and $\tau_{0}$ contains $C_{1}$ and $C_{2}$. 
\end{prop}
\begin{proof}
For simplicity, we prove the case that $C_{1}$ and $C_{2}$ are open components with distinct directions. 
By Lemma~\ref{lem:good_tp}, there is a good motif $L$ of $\tau$ with a good fundamental domain $F$. 
Suppose that for $i = 1,2$, $C_{i}$ is projected onto a component $K_{i}$ of $L$. 
Let $G_{i}$ be the subgroup of $\pi_{1} (T^{3}) \cong H_{1}(T^{3}; \bbZ) \cong \bbZ^{3}$ generated by $[K_{i}]$. 
Let $G$ be the subgroup of $\pi_{1} (T^{3}) \cong \bbZ^{3}$ generated by $G_{1}$ and $G_{2}$. 
Then $\bigcup_{g \in G} g(C_{1}) \cup g(C_{2})$ is the image of a doubly periodic tangle $\tau_{0}$ by an embedding $\iota$. 
\end{proof}

\begin{dfn}
Let $C_{1}$ and $C_{2}$ be two oriented components of a triply periodic tangle $\tau$. 
Suppose that a doubly periodic tangle $\tau_{0}$ contains $C_{1}$ and $C_{2}$. 
Then we define $\lk(C_{1}, C_{2})$ to be the linking number of $C_{1}$ and $C_{2}$ in $\tau_{0}$. 
\end{dfn}

\begin{thm}
\label{thm:lk_tp}
Let $C_{1}$ and $C_{2}$ be two oriented components of a triply periodic tangle $\tau$. 
Then the linking number $\lk(C_{1}, C_{2})$ does not depend on the choice of a doubly periodic tangle $\tau_{0}$ that contains $C_{1}$ and $C_{2}$. 
Moreover, if there is a link-homotopy from $\tau$ to a triply periodic tangle $\tau'$ which maps $C_{1}$ and $C_{2}$ respectively to components $C'_{1}$ and $C'_{2}$, 
then $\lk(C_{1}, C_{2}) = \lk(C'_{1}, C'_{2})$. 
In other words, the linking number is well-defined for a link-homotopy class of a triply periodic tangle. 
\end{thm}
\begin{proof}
Note that if we replace $\tau \subset \bbR^{2} \times I$ with its image by the map $(x, y, z) \mapsto (-x, y, 1-z)$ on $\bbR^{2} \times I$, called the reverse of $\tau$, 
then the linking number of corresponding components does not change. 
From now on, we do not distinguish $\tau$ and its reverse. 
Let $\tau_{0}$ and $\tau'_{0}$ be doubly periodic tangles that contain $C_{1}$ and $C_{2}$. 
If both $C_{1}$ and $C_{2}$ are closed, then $\lk(C_{1}, C_{2})$ coincides with the ordinary linking number in $\bbR^{3}$. 
If $C_{1}$ and $C_{2}$ are open components with distinct directions, 
then $\tau_{0}$ and $\tau'_{0}$ contain a common doubly periodic tangle consisting of translational images of $C_{1}$ and $C_{2}$. 
Otherwise $\tau_{0}$ and $\tau'_{0}$ contain a common singly periodic tangle consisting of translational images of $C_{1}$ and $C_{2}$. 
In all the cases, the linking number $\lk(C_{1}, C_{2})$ does not depend on the choice of $\tau_{0}$ and $\tau'_{0}$. 
In the last case, it follows from Proposition~\ref{prop:sp_in_dp}. 

For the second assertion, a doubly periodic tangle $\tau_{0}$ in $\tau$ containing $C_{1}$ and $C_{2}$ corresponds to a doubly periodic tangle $\tau'_{0}$ in $\tau'$ containing $C'_{1}$ and $C'_{2}$. 
Since $\tau_{0}$ and $\tau'_{0}$ are link-homotopic, 
the linking number is well-defined by Theorem~\ref{thm:lk_dp}.
\end{proof}

\section{Higher order linking numbers}
\label{section:higher}

In this section, we define the higher order linking numbers for periodic tangles under certain assumptions, as a generalization of the linking number.
In this section, for an oriented link $K$ or a string link $\sigma$, we denote by $\operatorname{lk}_K(i,j)$ or $\operatorname{lk}_\sigma(i,j)$ the linking number of the $i$th and $j$th components.

\subsection{Higher order linking numbers for links and string links}\label{subsec:def_Milnor}

The higher order linking numbers (or Milnor invariants) of links were introduced by Milnor \cite{Milnor1} (see also \cite{Milnor2}) to classify links up to link-homotopy. 
We recall the definition of Milnor invariants following \cite{Milnor2}; see \cite{Milnor2} for details.

Let $K$ be an $m$-component ordered oriented link in $S^3$. Denote by $\pi$ the fundamental group of the complement of $K$ in $S^3$, and by $\pi_q$ the $q$th term of the lower central series of $\pi$, where $\pi_1=\pi$ and $\pi_{q+1}=[\pi,\pi_q]$, which is the commutator subgroup of $\pi$ and $\pi_q$. Choose meridians $a_1,\dots,a_m$ and longitudes $\lambda_1,\dots,\lambda_m$ of $K$, where $a_i$ and $\lambda_i$ are a meridian and a longitude of the $i$th component, respectively and each $\lambda_i$ has exponent sum zero in $a_i$. It is known that $\pi/\pi_q$ is generated by the images of $a_1,\dots,a_m$, so that $\lambda_i$ is represented in $\pi/\pi_q$ by a word in $a_1,\dots,a_m$.
The Magnus expansion $M$ is the homomorphism from the free group generated by $a_1, \dots, a_m$ to the formal power series ring in non-commutative variables $X_1,\dots,X_m$ defined by
\[
M(a_j)=1+X_j,\qquad M(a_j^{-1})=1-X_j+X_j^2-X_j^3+\cdots
\qquad(1\leq j\leq m).
\]
For a sequence $I=i_1i_2\cdots i_ki$ with entries in $\{1,\dots,m\}$, we
denote by $\mu_K(I)$ the coefficient of the monomial $X_{i_1}X_{i_2}\cdots X_{i_k}$
in $M(\lambda_i)$, where $q$ is chosen to be greater than $k$. 
This integer depends on the choices of the meridians and the longitudes, and the \emph{Milnor invariant} $\overline{\mu}_K(I)$ is defined as the residue class
of $\mu_K(I)$ modulo $\Delta_K(I)$, the greatest common divisor of all $\mu_K(J)$ such that $J$ is obtained from $I$ by removing at least one index and permuting the remaining indices cyclically~\cite{Milnor2}.
We call the number $|I|$ of indices in $I$ the \emph{length} of $\overline{\mu}_K(I)$. The residue class $\overline{\mu}_K(I)$ is invariant under ambient isotopy for any sequence $I$, and under link-homotopy for any non-repeating sequence $I$, which is called the \emph{Milnor link-homotopy invariant}. 
Note that $\overline{\mu}_K(ij)$ is the linking
number of the $i$th and $j$th components, and that Milnor invariants can be computed from a diagram of $K$.

A \emph{string link} is a pure tangle in $D^2\times[0,1]$ without closed
components whose strings are oriented upward, from $D^2\times\{0\}$ to
$D^2\times\{1\}$, as defined by Habegger--Lin~\cite{HL} and illustrated in
Figure~\ref{fig:stringlink}. For two $m$-component string links $\sigma$ and
$\sigma'$, the \emph{stacking product} $\sigma\cdot\sigma'$ is the
$m$-component string link obtained by stacking $\sigma'$ on top of $\sigma$.
We take the base point on the boundary of $D^2\times\{0\}$ behind the strings, and let
$a_1,\dots,a_m$ be the meridians on $D^2\times\{0\}$. 
The longitude $\lambda_i$ of the $i$th string is the class of the loop that goes from the base point to the bottom of the $i$th string, runs upward along a parallel of it, and returns to the base point along the boundary. 
Changing the parallel replaces $\lambda_i$ by $a_i^n\lambda_i$ for some $n\in\mathbb{Z}$, which
changes the exponent sum of $a_i$ in $\lambda_i$ by $n$. Hence there is a unique choice of the parallel such that the exponent sum of $a_i$ in $\lambda_i$ is zero, and we always take this choice.

Then the meridian of the $i$th string on $D^2\times\{1\}$ is $\lambda_i^{-1}a_i\lambda_i$. 
Since the nilpotent quotients of the
fundamental group of the complement of $\sigma$ are isomorphic to those of the free group
generated by $a_1,\dots,a_m$~\cite{HL}, defining $\mu_\sigma(I)$ as above
gives an integer, called a \emph{Milnor $\mu$-invariant} of $\sigma$;
since the meridians and the longitudes are determined by the boundary, there is no indeterminacy.
It is invariant under ambient isotopy fixing the boundary for any sequence $I$, and under link-homotopy fixing the boundary for any non-repeating
sequence $I$, which is called the \emph{Milnor link-homotopy invariant}.
The Milnor link-homotopy invariants of string links form a complete invariant up to link-homotopy \cite{HL}.

\begin{figure}[h]
  \centering
  \setlength{\unitlength}{0.004\textwidth}
  \begin{picture}(108,92)
    \put(0,0){\includegraphics[width=0.4\textwidth]{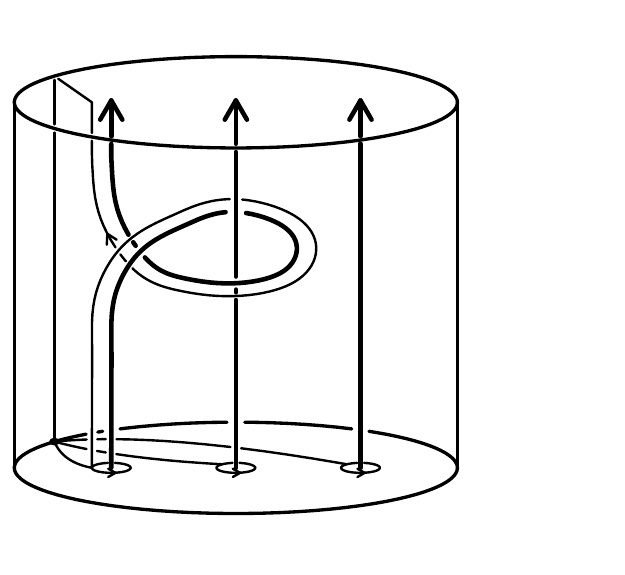}}
    \put(17.8,87.8){\makebox(0,0){$1$}}
    \put(37.7,87.8){\makebox(0,0){$2$}}
    \put(57.7,87.8){\makebox(0,0){$3$}}
    \put(75.8,75.7){\makebox(0,0)[l]{$D^2\times\{1\}$}}
    \put(75.8,17.2){\makebox(0,0)[l]{$D^2\times\{0\}$}}
    \put(18.3,6.3){\makebox(0,0){$a_1$}}
    \put(38.2,6.3){\makebox(0,0){$a_2$}}
    \put(58.1,6.3){\makebox(0,0){$a_3$}}
    \put(11.4,31.0){\makebox(0,0){$\lambda_1$}}
    \put(4.9,24.1){\makebox(0,0){$\ast$}}
  \end{picture}
  \caption{A string link with the base point $\ast$ on the boundary of $D^2\times\{0\}$, the meridians $a_1,a_2,a_3$, and the longitude $\lambda_1$}
  \label{fig:stringlink}
\end{figure}

Meilhan--Yasuhara \cite{MY} showed the following additivity property.
\begin{lem}[{\cite[Lemma 3.3]{MY}}]\label{additive}
    Let $\sigma$ and $\sigma'$ be two $m$-component string links such that all Milnor invariants of $\sigma$ (resp. $\sigma'$) of length $\leq l$ (resp. $\leq l'$) vanish.
Then $\mu_{\sigma\cdot\sigma'}(I)=\mu_\sigma(I)+\mu_{\sigma'}(I)$
for every sequence $I$ of length at most $l+l'$.
\end{lem}

Moreover, we prepare several lemmas on Milnor link-homotopy invariants, which will be used in Subsection~\ref{sec:Milnorfordp}.

For an $m$-component string link $\sigma$, let $\varphi_\sigma$ be the endomorphism of the
free group generated by $a_1,\dots,a_m$ that sends each meridian on
$D^2\times\{1\}$ to its expression in $a_1,\dots,a_m$; thus
$\varphi_\sigma(a_i)=\lambda_i^{-1}a_i\lambda_i$ and
$\varphi_{\sigma\cdot\tau}=\varphi_\sigma\circ\varphi_\tau$.

\begin{lem}\label{product-mu123}
Let $\sigma$ and $\tau$ be 3-component string links. Then
\[
\mu_{\sigma\cdot\tau}(123)=\mu_\sigma(123)+\mu_\tau(123)
+\mathrm{lk}_\sigma(1,3)\,\mathrm{lk}_\tau(2,3)
+\mathrm{lk}_\sigma(1,2)\,\mathrm{lk}_\tau(1,3)
-\mathrm{lk}_\sigma(1,2)\,\mathrm{lk}_\tau(2,3).
\]
\end{lem}
\begin{proof}
Let $\lambda_i^\sigma$ and $\lambda_i^\tau$ be the longitudes of $\sigma$ and
$\tau$. Since the third string passes through $\sigma$ and then through
$\tau$, we have
$\lambda_3^{\sigma\cdot\tau}=\lambda_3^\sigma\cdot\varphi_\sigma(\lambda_3^\tau)$ and
\[ M(\lambda_3^{\sigma\cdot\tau})=M(\lambda_3^\sigma)\,M(\varphi_\sigma(\lambda_3^\tau)).
\]
Since $\varphi_\sigma(a_i)=(\lambda_i^\sigma)^{-1}a_i\lambda_i^\sigma$, we have
\[
  M(\varphi_\sigma(a_i))
  =M((\lambda_i^\sigma)^{-1})\,(1+X_i)\,M(\lambda_i^\sigma)
  =1+M((\lambda_i^\sigma)^{-1})\,X_i\,M(\lambda_i^\sigma).
\]
Let $L_i^\sigma$ and $L_i^\tau$ be the linear parts of $M(\lambda_i^\sigma)$ and $M(\lambda_i^\tau)$.
Because $M(\lambda_i^\sigma)=1+L_i^\sigma+(\text{degree}\ge 2)$ and $M((\lambda_i^\sigma)^{-1})=1-L_i^\sigma+(\text{degree}\ge 2)$, this gives
\[
  M(\varphi_\sigma(a_i))=1+X_i+[X_i,L_i^\sigma]+(\text{degree}\ge 3),
\]
where $[X,Y]=XY-YX$. 
Hence $M(\varphi_\sigma(\lambda_3^\tau))$ is obtained from $M(\lambda_3^\tau)$ by substituting $X_i+[X_i,L_i^\sigma]+(\text{degree}\ge 3)$ for each $X_i$.
Since 
$L_1^\sigma=\mathrm{lk}_\sigma(1,2)X_2+\mathrm{lk}_\sigma(1,3)X_3$,
$L_2^\sigma=\mathrm{lk}_\sigma(1,2)X_1+\mathrm{lk}_\sigma(2,3)X_3$,
the coefficients of $X_1X_2$ in $[X_1,L_1^\sigma]$ and $[X_2,L_2^\sigma]$ are $\mathrm{lk}_\sigma(1,2)$ and $-\mathrm{lk}_\sigma(1,2)$.
Therefore, by  $L_3^\tau=\mathrm{lk}_\tau(1,3)X_1+\mathrm{lk}_\tau(2,3)X_2$, the substitution changes the coefficient of $X_1X_2$ by
$\mathrm{lk}_\sigma(1,2)\mathrm{lk}_\tau(1,3)-\mathrm{lk}_\sigma(1,2)\mathrm{lk}_\tau(2,3)$.
Since $L_3^\sigma=\mathrm{lk}_\sigma(1,3)X_1+\mathrm{lk}_\sigma(2,3)X_2$, the product of the linear parts of $M(\lambda_3^\sigma)$ and
$M(\varphi_\sigma(\lambda_3^\tau))$, whose linear part is $L_3^\tau$, contributes
$\mathrm{lk}_\sigma(1,3)\mathrm{lk}_\tau(2,3)$.

Therefore, the coefficient of $X_1X_2$ in
$M(\lambda_3^{\sigma\cdot\tau})=M(\lambda_3^\sigma)\,M(\varphi_\sigma(\lambda_3^\tau))$
is the sum of the following four terms:
the coefficient $\mu_\sigma(123)$ of $X_1X_2$ in the quadratic part of
$M(\lambda_3^\sigma)$;
the coefficient $\mu_\tau(123)$ of $X_1X_2$ in the quadratic part of
$M(\lambda_3^\tau)$, which is not affected by the substitution modulo terms
of degree at least $3$;
the change
$\mathrm{lk}_\sigma(1,2)\,\mathrm{lk}_\tau(1,3)-\mathrm{lk}_\sigma(1,2)\,\mathrm{lk}_\tau(2,3)$
caused by the substitution;
and the contribution $\mathrm{lk}_\sigma(1,3)\,\mathrm{lk}_\tau(2,3)$ of the
product of the linear parts.
This gives the desired formula.
\end{proof}

\begin{lem}\label{lem:132}
For a 3-component string link $\sigma$,
\[
\mu_\sigma(132)=-\mu_\sigma(123)+\mathrm{lk}_\sigma(1,2)\,\mathrm{lk}_\sigma(1,3)+\mathrm{lk}_\sigma(1,3).
\]
\end{lem}
\begin{proof}
Let $L_i$ and $Q_i$ be the linear and quadratic parts of $M(\lambda_i)$. Then $L_i=\sum_{j\neq i}\mathrm{lk}_\sigma(i,j)X_j$, and for $j,k\neq i$,
the coefficient of $X_jX_k$ in $Q_i$ is $\mu_\sigma(jki)$.
Moreover, 
$$M(\lambda_i)^{-1}=1-L_i+(L_i^2-Q_i)+(\text{degree}\ge 3).$$
Therefore, 
\[
M(\varphi_\sigma(a_i))=M(\lambda_i^{-1}a_i\lambda_i)=1+X_i+D_i+C_i+(\text{degree}\ge 4),
\]
where
$D_i=X_iL_i-L_iX_i$ and $C_i=X_iQ_i-Q_iX_i-L_iX_iL_i+L_i^2X_i$.

Since $a_1a_2a_3$ is represented by a loop parallel to the boundary of $D^2$, we have $\varphi_\sigma(a_1a_2a_3)=a_1a_2a_3$. We compare the coefficients of $X_1X_2X_3$ in
$M(\varphi_\sigma(a_1a_2a_3))=M(a_1a_2a_3)$.
The part of degree $3$ of $M(\varphi_\sigma(a_1a_2a_3))$ is
\[
  X_1X_2X_3+C_1+C_2+C_3+X_1D_2+D_1X_2+X_1D_3+D_1X_3+X_2D_3+D_2X_3 .
\]
The coefficients of $X_1X_2X_3$ in $C_1$, $C_2$ and $C_3$ are
$\mu_\sigma(231)$, $-\mathrm{lk}_\sigma{(1,2)}\mathrm{lk}_\sigma{(2,3)}$ and $-\mu_\sigma(123)+\mathrm{lk}_\sigma{(1,3)}\mathrm{lk}_\sigma{(2,3)}$,
respectively, and those in $X_1D_2$, $X_1D_3$, $D_1X_3$ and $D_2X_3$ are
$\mathrm{lk}_\sigma{(2,3)}$, $-\mathrm{lk}_\sigma{(2,3)}$, $\mathrm{lk}_\sigma{(1,2)}$ and $-\mathrm{lk}_\sigma{(1,2)}$, respectively, while $D_1X_2$ and
$X_2D_3$ do not contribute. Since the coefficient of $X_1X_2X_3$ in
$(1+X_1)(1+X_2)(1+X_3)$ is $1$, we obtain
\begin{align}
\mu_\sigma(231)=\mu_\sigma(123)+\mathrm{lk}_\sigma{(1,2)}\mathrm{lk}_\sigma{(2,3)}-\mathrm{lk}_\sigma{(1,3)}\mathrm{lk}_\sigma{(2,3)}. \label{eq:mu} 
\end{align}

Similarly, comparing the coefficients of $X_1X_3X_2$ in
$M(\varphi_\sigma(a_1a_2a_3))=M(a_1a_2a_3)$ gives
\[
\mu_\sigma(132)=\mu_\sigma(321)+lk{(1,2)}lk{(2,3)}-lk{(1,3)}lk{(2,3)}+lk{(1,3)}.
\]
Together with the shuffle relation for the Magnus expansion of $\lambda_1$ by Chen–Fox–Lyndon \cite[Lemma~3.3]{CFL}, $\mu_\sigma(231)+\mu_\sigma(321)=\mathrm{lk}_\sigma{(1,2)}\mathrm{lk}_\sigma{(1,3)}$, we obtain the formula.
\end{proof}

For the positive crossing $\sigma_i$ of the $i$th and $(i+1)$st strings, where $\sigma_i$ denotes the standard generator of the braid group, 
$\varphi_{\sigma_i}(a_i)=a_{i+1}$,
$\varphi_{\sigma_i}(a_{i+1})=a_{i+1}^{-1}a_ia_{i+1}$, and $\varphi_{\sigma_i}$
fixes the other generators.

\begin{lem}\label{half-twist-mu123}
Let $\sigma$ be a 3-component string link, let $\varepsilon\in\{1,-1\}$.
Let $\sigma'=\sigma_2^{-\varepsilon}\cdot\sigma\cdot\sigma_2^{\varepsilon}$.
Then
\[
  \mathrm{lk}_{\sigma'}(1,2)=\mathrm{lk}_\sigma(1,3),\quad
  \mathrm{lk}_{\sigma'}(1,3)=\mathrm{lk}_\sigma(1,2),\quad
  \mathrm{lk}_{\sigma'}(2,3)=\mathrm{lk}_\sigma(2,3),
\]
and
\[
\mu_{\sigma'}(123)=-\mu_\sigma(123)+\mathrm{lk}_\sigma(1,2)\,\mathrm{lk}_\sigma(1,3)+
\begin{cases} \mathrm{lk}_\sigma(1,2) & (\varepsilon=1),\\
\mathrm{lk}_\sigma(1,3) & (\varepsilon=-1).\end{cases}
\]
\end{lem}

\begin{proof}
The first assertion is obvious.
First, we calculate $\mu_{\sigma'}(231)$.
Since the first string of $\sigma_2^{\pm1}$ is trivial, we have
$\lambda_1^{\sigma_2^{\pm1}}=1$, and hence
\[
  \lambda_1^{\sigma'}
  =\lambda_1^{\sigma_2^{-\varepsilon}}\cdot
   \varphi_{\sigma_2^{-\varepsilon}}(\lambda_1^\sigma)\cdot
   \varphi_{\sigma_2^{-\varepsilon}\cdot\sigma}(\lambda_1^{\sigma_2^{\varepsilon}})
  =\varphi_{\sigma_2^{-\varepsilon}}(\lambda_1^\sigma).
\]
We have
\[
  M(\lambda_1^\sigma)=1+\mathrm{lk}_\sigma(1,2)X_2+\mathrm{lk}_\sigma(1,3)X_3
  +Q_1+(\text{degree}\ge 3),
\]
where the coefficients of $X_2X_3$ and $X_3X_2$ in $Q_1$ are $\mu_\sigma(231)$ and $\mu_\sigma(321)$, respectively.
If $\varepsilon=1$, then $\varphi_{\sigma_2^{-1}}(a_2)=a_2a_3a_2^{-1}$ and $\varphi_{\sigma_2^{-1}}(a_3)=a_2$, so $M(\varphi_{\sigma_2^{-1}}(w))$ is obtained from $M(w)$ by the substitution $X_2\mapsto X_3+X_2X_3-X_3X_2+(\text{degree}\ge 3)$, $X_3\mapsto X_2$;
the only terms of $M(\lambda_1^\sigma)$ contributing to the coefficient of $X_2X_3$ are $\mathrm{lk}_\sigma(1,2)X_2$ and $\mu_\sigma(321)X_3X_2$, which give $\mu_{\sigma'}(231)=\mathrm{lk}_\sigma(1,2)+\mu_\sigma(321)$.
If $\varepsilon=-1$, then similarly, we obtain
$\mu_{\sigma'}(231)=\mathrm{lk}_\sigma(1,3)+\mu_\sigma(321)$.

We now calculate $\mu_{\sigma'}(123)$ from $\mu_{\sigma'}(231)$.
The shuffle relation for the Magnus expansion of $\lambda_1^\sigma$
\cite{CFL} gives
\begin{align} \mu_\sigma(321)=\mathrm{lk}_\sigma(1,2)\,\mathrm{lk}_\sigma(1,3)-\mu_\sigma(231).\label{eq:321} 
\end{align}
By the equalities (\ref{eq:mu}) in the proof of
Lemma~\ref{lem:132} and (\ref{eq:321}), we have
\[
  \mu_\sigma(321)=\mathrm{lk}_\sigma(1,2)\,\mathrm{lk}_\sigma(1,3)-\mu_\sigma(123)
  -\mathrm{lk}_\sigma(1,2)\,\mathrm{lk}_\sigma(2,3)
  +\mathrm{lk}_\sigma(1,3)\,\mathrm{lk}_\sigma(2,3).
\]
Similarly, applying (\ref{eq:mu}) to $\sigma'$ and using the first assertion, we have
\[
  \mu_{\sigma'}(123)=\mu_{\sigma'}(231)-\mathrm{lk}_\sigma(1,3)\,\mathrm{lk}_\sigma(2,3)
  +\mathrm{lk}_\sigma(1,2)\,\mathrm{lk}_\sigma(2,3).
\]
Finally, put $c=\mathrm{lk}_\sigma(1,2)$ if $\varepsilon=1$ and 
$c=\mathrm{lk}_\sigma(1,3)$ if $\varepsilon=-1$, so that $\mu_{\sigma'}(231)=c+\mu_\sigma(321)$ as shown above. 
Substituting the expression of $\mu_{\sigma}(321)$, we obtain
\begin{align*}
  \mu_{\sigma'}(123)
  &=c+\mu_\sigma(321)-\mathrm{lk}_\sigma(1,3)\,\mathrm{lk}_\sigma(2,3)
   +\mathrm{lk}_\sigma(1,2)\,\mathrm{lk}_\sigma(2,3)\\
  &=c+\mathrm{lk}_\sigma(1,2)\,\mathrm{lk}_\sigma(1,3)-\mu_\sigma(123),
\end{align*}
which is the desired formula.
\end{proof}

We define the Milnor link-homotopy invariant for a 2-component string link with a closed loop.    

\begin{dfn}\label{dfn:case3}
Let $\sigma\subset D^2\times[0,1]$ be the union of an oriented 2-component string link $\sigma_{ij}$, whose strings are the $i$th and $j$th components from left to right, and an oriented closed loop $K_k$ as the $k$th component, where $i<j$ and $\{i,j,k\}=\{1,2,3\}$. 
Put $X=D^2\times[0,1]\setminus\sigma$, and take the base point $\ast$ on $\partial D^2\times\{0\}$. Let $a_i$ and $a_j$ be the meridians at the bottom of the $i$th and $j$th components as above, and let $a_k$ and $\lambda_k$ be a meridian and a longitude of $K_k$ connected to $\ast$ by a path $\gamma$ in $X$, where the longitude is normalized as above. 
Then $\pi_1(X)/\pi_1(X)_q$ is generated by $a_i,a_j,a_k$ for $q\ge 3$ \cite{Milnor2, Stallings}.
Let $\mu_\sigma(ijk)$ and $\mu_\sigma(jik)$ be the coefficients of $X_iX_j$ and $X_jX_i$ in $M(\lambda_k)$, respectively, computed by lifting $\lambda_k$ to the free group $F(a_i,a_j,a_k)$. We define
\[
  \overline{\mu}_\sigma(ijk)
  =\mu_\sigma(ijk)\bmod\gcd\{\mathrm{lk}_\sigma(i,k),\mathrm{lk}_\sigma(j,k)\},
\]
and 
\[
  \overline{\mu}_\sigma(jik)
  =\mu_\sigma(jik)\bmod\gcd\{\mathrm{lk}_\sigma(i,k),\mathrm{lk}_\sigma(j,k)\}.
\]
\end{dfn}

\begin{prop}\label{prop:case3-invariance}
In the setting of Definition~\ref{dfn:case3}, the residue classes $\overline{\mu}_\sigma(ijk)$ and $\overline{\mu}_\sigma(jik)$ do not depend on the choices of the path $\gamma$ and the longitude $\lambda_k$, and are invariant under link-homotopy fixing the endpoints of the $i$th and $j$th components. Moreover, $\overline{\mu}_\sigma(jik)=-\overline{\mu}_\sigma(ijk)$.
\end{prop}

\begin{proof}
Put $d=\gcd\{\mathrm{lk}_\sigma(i,k),\mathrm{lk}_\sigma(j,k)\}$ and $Y=D^2\times[0,1]\setminus\sigma_{ij}$. 
For a group $G$ arising as the fundamental group of the complement of oriented arcs or loops, we denote by $\mathcal{R}G$ the reduced group of $G$, that is, the quotient of $G$ by the normal subgroup generated by $[m,hmh^{-1}]$ for all $h\in G$ and all meridians $m$ of the components \cite{Milnor1}. 
Since the meridians of each component form a conjugacy class up to inverses, $\mathcal{R}G$ does not depend on the choice of meridians. For the free group $F(a_i,a_j)$, we regard $a_i$ and $a_j$ as meridians of $i$th and $j$th components.

Let $L_k=\mathrm{lk}_\sigma(i,k)X_i+\mathrm{lk}_\sigma(j,k)X_j$ be the linear part of $M(\lambda_k)$. 
The natural homomorphism from $F(a_i,a_j,a_k)/\langle\!\langle[a_k,\lambda_k]\rangle\!\rangle$ to $\pi_1(X)$ induces an isomorphism on $H_1$ and a surjection on $H_2$, since $H_2(X)$ is generated by the boundary torus of $K_k$. 
Hence, by Stallings' theorem \cite{Stallings}, as in the case of links in $S^3$ \cite{Milnor2}, $\pi_1(X)/\pi_1(X)_q$ is the quotient of $F(a_i,a_j,a_k)$ by the normal closure of $[a_k,\lambda_k]$ and $F(a_i,a_j,a_k)_q$. 
Therefore, two lifts of $\lambda_k$ to $F(a_i,a_j,a_k)$ differ by a product of conjugates of $[a_k,\lambda_k]^{\pm1}$ and an element of $F(a_i,a_j,a_k)_q$, whose Magnus expansion is of the form $1+(\text{degree}\ge q)$.
The series $M([a_k,\lambda_k])-1$ has no linear part, and its quadratic part is $[X_k,L_k]$, every term of which contains $X_k$. 
Hence, for fixed choices of $\gamma$ and $\lambda_k$, the integers $\mu_\sigma(ijk)$ and $\mu_\sigma(jik)$ do not depend on the choice of the lift of $\lambda_k$.

We first show that $\overline{\mu}_\sigma(ijk)$ does not depend on the choices of the path $\gamma$ and the longitude $\lambda_k$.
The inclusion $X\subset Y$ induces a homomorphism $p\colon\pi_1(X)\to\pi_1(Y)$ which sends $a_i$ and $a_j$ to themselves and $a_k$ to $1$, since $a_k$ bounds a meridian disk of $K_k$ in $Y$.
By \cite[Section~1]{HL}, the inclusion of the bottom induces isomorphisms
\[
  F(a_i,a_j)/F(a_i,a_j)_q\cong\pi_1(Y)/\pi_1(Y)_q,\qquad
  \mathcal{R}F(a_i,a_j)\cong\mathcal{R}\pi_1(Y).
\]
On the level of free groups, $p$ corresponds to the homomorphism $F(a_i,a_j,a_k)\to F(a_i,a_j)$ sending $a_k$ to $1$, and hence on Magnus expansions to the substitution $X_k\mapsto 0$, which does not change the coefficient of $X_iX_j$. 
Moreover, $p(\lambda_k)$ is represented by $K_k$ itself, connected to $\ast$ by $\gamma$, since $\lambda_k$ is a parallel of $K_k$.
Therefore $\mu_\sigma(ijk)$ is equal to the coefficient of $X_iX_j$ in $M(w)$, where $w\in F(a_i,a_j)$ represents the element of $\pi_1(Y)/\pi_1(Y)_3$ given by $K_k$ and $\gamma$, and the linear part of $M(w)$ is $L_k$.
If we replace $\lambda_k$ by $\lambda_ka_k^n$, the element $w$ does not change, since $p(a_k)=1$. 
If we replace $\gamma$ by another path, $w$ is replaced by a conjugate $hwh^{-1}$ with $h\in F(a_i,a_j)$. 
Let $h_i$ and $h_j$ be the coefficients of $X_i$ and $X_j$ in $M(h)$. Then the coefficient of $X_iX_j$ changes by the coefficient of $X_iX_j$ in $[h_iX_i+h_jX_j,L_k]$, namely
$h_i\,\mathrm{lk}_\sigma(j,k)-h_j\,\mathrm{lk}_\sigma(i,k)$, which is divisible
by $d$.
Hence $\overline{\mu}_\sigma(ijk)$ does not depend on the choices, and depends only on the image of $K_k$ in the set of conjugacy classes of $\mathcal{R}\pi_1(Y)\cong\mathcal{R}F(a_i,a_j)$; 
here we use that the coefficient of a word without repeated letters is well defined on the reduced free group \cite{Milnor1,HL}.

Next, we show that $\overline{\mu}_\sigma(ijk)$ is invariant under link-homotopy fixing the endpoints of the $i$th and $j$th components. 
Such a link-homotopy is a composition of ambient isotopies fixing $\partial(D^2\times[0,1])$ and self-crossing changes.
Since linking numbers are invariant under link-homotopy, $d$ does not change.
An ambient isotopy fixing the boundary induces an isomorphism of the fundamental groups preserving $a_i$ and $a_j$ and carrying $K_k$ to the new $k$th component, so $\overline{\mu}_\sigma(ijk)$ does not change, since it does not depend on the choices of the path and the longitude, as shown above.
A self-crossing change of $K_k$ is a homotopy of $K_k$ in $Y$, so it does not change $\overline{\mu}_\sigma(ijk)$, since $\overline{\mu}_\sigma(ijk)$ depends only on the free homotopy class of $K_k$ in $Y$, as shown above.
Suppose that $\sigma'$ is obtained from $\sigma$ by a self-crossing change of the $i$th or $j$th component in a small ball $B$ disjoint from $K_k$ and from $\partial(D^2\times[0,1])$, and let $Y'$ be the complement of the resulting string link $\sigma'_{ij}$. 
By the argument of Milnor \cite[Section~2]{Milnor1} (see also \cite[Section~1]{HL}), the inclusions $Y\setminus B\subset Y$ and $Y\setminus B=Y'\setminus B\subset Y'$ induce an isomorphism $\Phi\colon\mathcal{R}\pi_1(Y)\to\mathcal{R}\pi_1(Y')$ such that $\Phi([c]_Y)=[c]_{Y'}$ for every loop $c$ in $Y\setminus B$ based at $\ast$.
In particular, $\Phi$ is compatible with the identifications $\mathcal{R}\pi_1(Y)\cong\mathcal{R}F(a_i,a_j)\cong\mathcal{R}\pi_1(Y')$. Since $K_k$ and $\gamma$ can be taken in $Y\setminus B$, the images of $K_k$ in $\mathcal{R}F(a_i,a_j)$ coincide, and hence $\overline{\mu}_{\sigma'}(ijk)=\overline{\mu}_\sigma(ijk)$, since $\overline{\mu}_\sigma(ijk)$ depends only on the image of $K_k$ in $\mathcal{R}F(a_i,a_j)$.

Finally, by the shuffle relation for the Magnus expansion of $\lambda_k$ \cite{CFL}, we have $\mu_\sigma(ijk)+\mu_\sigma(jik)=\mathrm{lk}_\sigma(i,k)\,\mathrm{lk}_\sigma(j,k)$, which is divisible by $d$. Hence $\overline{\mu}_\sigma(jik)=-\overline{\mu}_\sigma(ijk)$, and the assertions for $\overline{\mu}_\sigma(jik)$ follow from those for $\overline{\mu}_\sigma(ijk)$.
\end{proof}

\begin{lem}\label{standcal}
Let $x_{12}=\sigma_1^2$, $x_{23}=\sigma_2^2$, and
$x_{13}=\sigma_2\cdot\sigma_1^2\cdot\sigma_2^{-1}$ be the standard generators of
the pure braid group, and write $[g,h]=g\cdot h\cdot g^{-1}\cdot h^{-1}$. Note
that the closure of $[x_{23},x_{12}]$ is the Borromean rings. For integers
$a,b,c,d$, let
\[
  \sigma(a,b,c,d)=x_{12}^{a}\cdot x_{13}^{b}\cdot x_{23}^{c}\cdot[x_{23},x_{12}]^{d}.
\]
Then we have $\mathrm{lk}(1,2)=a$, $\mathrm{lk}(1,3)=b$, $\mathrm{lk}(2,3)=c$, and
\[
  \mu_{\sigma(a,b,c,d)}(123)=ab+bc-ca+d.
\]
Moreover, any 3-component string link is link-homotopic to $\sigma(a,b,c,d)$
fixing the boundary for unique integers $a,b,c,d$.
\end{lem}

\begin{proof}
The linking numbers are clear. We have $\mu(123)=0$ for $x_{12}^a$,
$x_{13}^b$ and $x_{23}^c$, and applying Lemma~\ref{product-mu123} to
$x_{23}\cdot x_{12}\cdot x_{23}^{-1}\cdot x_{12}^{-1}$ from the left gives
$\mu_{[x_{23},x_{12}]}(123)=1$. Since all the linking numbers of
$[x_{23},x_{12}]$ vanish, Lemma~\ref{product-mu123} gives
$\mu_{[x_{23},x_{12}]^d}(123)=d$. Applying Lemma~\ref{product-mu123}
successively to the product
$x_{12}^{a}\cdot x_{13}^{b}\cdot x_{23}^{c}\cdot[x_{23},x_{12}]^{d}$, we obtain
the formula for $\mu_{\sigma(a,b,c,d)}(123)$.

For the last assertion, recall that string links are classified up to
link-homotopy by their Milnor invariants with non-repeating sequences
\cite{HL}. For a 3-component string link, the Milnor invariants of length
three are determined by $\mu(123)$ and the linking numbers, by the shuffle
relations, Lemma~\ref{lem:132} and the equality \eqref{eq:mu} in its proof.
Hence the link-homotopy class of a 3-component string link $\sigma$ is
determined by $(\mathrm{lk}_\sigma(1,2),\mathrm{lk}_\sigma(1,3),
\mathrm{lk}_\sigma(2,3),\mu_\sigma(123))\in\mathbb{Z}^4$. Since the map
$(a,b,c,d)\mapsto(a,b,c,ab+bc-ca+d)$ is a bijection of $\mathbb{Z}^4$, the
assertion follows.
\end{proof}

\subsection{Higher order linking numbers for singly periodic tangles}
Let $K=K_1\cup\cdots\cup K_m$ be an oriented $m$-component link in the solid torus $S^1\times D^2$, and let $\iota\colon S^1\times D^2\to\mathbb{R}^3$ be the standard embedding. 
The standard projection gives a diagram of $K$ on $S^1\times I$ in general position, which can be regarded as a diagram of the link $\iota(K)$ in $\mathbb{R}^3$. 
For a sequence $J$ on $\{1,\dots,m\}$, we denote by $\overline{\mu}_K(J)$ the Milnor invariant of $\iota(K)$, and by $\Delta_K(J)$ its indeterminacy as in Subsection~\ref{subsec:def_Milnor}; thus $\overline{\mu}_K(J)$ is an element of $\mathbb{Z}/\Delta_K(J)\mathbb{Z}$, and it can be computed from the above diagram.

Let $C$ be the union of $m$ oriented components $C_1, \dots, C_m$ of a singly periodic tangle $\tau$. 
The higher order linking number of $C$, denoted by $\overline{\mu}_C(J)$, is defined as follows. 
By Lemma~\ref{lem:good_sp} and Lemma~\ref{lem:very_good_sp}, there is a very good motif $L$ of $\tau$ for $C$. 

\begin{dfn}\label{dfn:higher-singly}
Let $L$ be a very good motif of a singly periodic tangle $\tau$ for the union $C$ of oriented components $C_1,\dots,C_m$.
Suppose that $C_i$ are respectively projected onto components $K_i$ of $L$. 
Let $K=K_1\cup\dots\cup K_m$.
Suppose that $K$ satisfies the following conditions:
\begin{itemize}
\item If all $C_i$ are open, then $\overline{\mu}_{K}(J')=0$ for all non-repeating sequences $J'$ of length at most $\lfloor\frac{m+1}{2}\rfloor$.
\item Otherwise, $\overline{\mu}_{K}(J')=0$ for all non-repeating sequences $J'$ not containing the indices of the closed components.
\end{itemize}
Then, for a non-repeating sequence $J$ on $\{1,\dots,m\}$, we define the higher
order linking number $\overline{\mu}_C(J)$ as follows. If all $C_i$ are open,
we define
\[
\overline{\mu}_C(J)=
\begin{cases}
\dfrac{\overline{\mu}_{K}(J)}{\Delta_{K}(J)}\in\mathbb{Q}/\mathbb{Z}
  & \text{if } \Delta_{K}(J)\neq 0,\\[2mm]
+\infty & \text{if } \Delta_{K}(J)=0 \text{ and } \overline{\mu}_{K}(J)>0,\\
-\infty & \text{if } \Delta_{K}(J)=0 \text{ and } \overline{\mu}_{K}(J)<0,\\
0 & \text{if } \Delta_{K}(J)=0 \text{ and } \overline{\mu}_{K}(J)=0.
\end{cases}
\]
Otherwise, we define
$\overline{\mu}_C(J)=\overline{\mu}_{K}(J)\in\mathbb{Z}/\Delta_{K}(J)\mathbb{Z}$.
\end{dfn}

Note that if all $C_i$ are open and all Milnor link-homotopy invariants of $K$ of length less than the length of $J$ vanish, then $\Delta_{K}(J)=0$, and hence $\overline{\mu}_C(J)\in\{0,\pm\infty\}$ is determined by the sign of the integer $\overline{\mu}_{K}(J)$. 
The second condition in Definition~\ref{dfn:higher-singly} implies that the string link formed by the open components is link-homotopically trivial \cite{HL}.

Let $J$ be a non-repeating sequence on $\{1,\dots,m\}$.
If the orientation of $C_i$ is reversed, then $\overline{\mu}_C(J)$ is replaced by $-\overline{\mu}_C(J)$ when $i$ appears in $J$, and does not change otherwise.
If $\tau$ is replaced by its mirror image, then $\overline{\mu}_C(J)$ is multiplied by $(-1)^{|J|-1}$.
In particular, the linking number changes its sign, while the triple linking number does not change.
Here we put $-(\pm\infty)=\mp\infty$.

\begin{ex}\label{exa:triple-singly}
For each of the four singly periodic tangles in Figure~\ref{fig:3rd_linking}, let $C=C_1\cup C_2\cup C_3$ be the union of the three components shown there, whose linking numbers all vanish.
In the first three examples, all of $C_1$, $C_2$ and $C_3$ are open, and hence
$\Delta_{K}(123)=0$. 
Their triple linking numbers $\overline{\mu}_C(123)$ are $-\infty$, $+\infty$ and $0$, respectively, and thus these three tangles are distinguished.
In the fourth example, its motif in Figure~\ref{fig:3rd_linking} is not good for $C$.
Therefore, we use its 2-sheeted cover, which is very good for $C$. 
Since $\Delta_{K}(123)=0$, the triple linking number is an integer, and we have $\overline{\mu}_C(123)=1$.
   \begin{figure}[h]
  \centering
  \includegraphics[angle=90, width=0.7\textwidth]{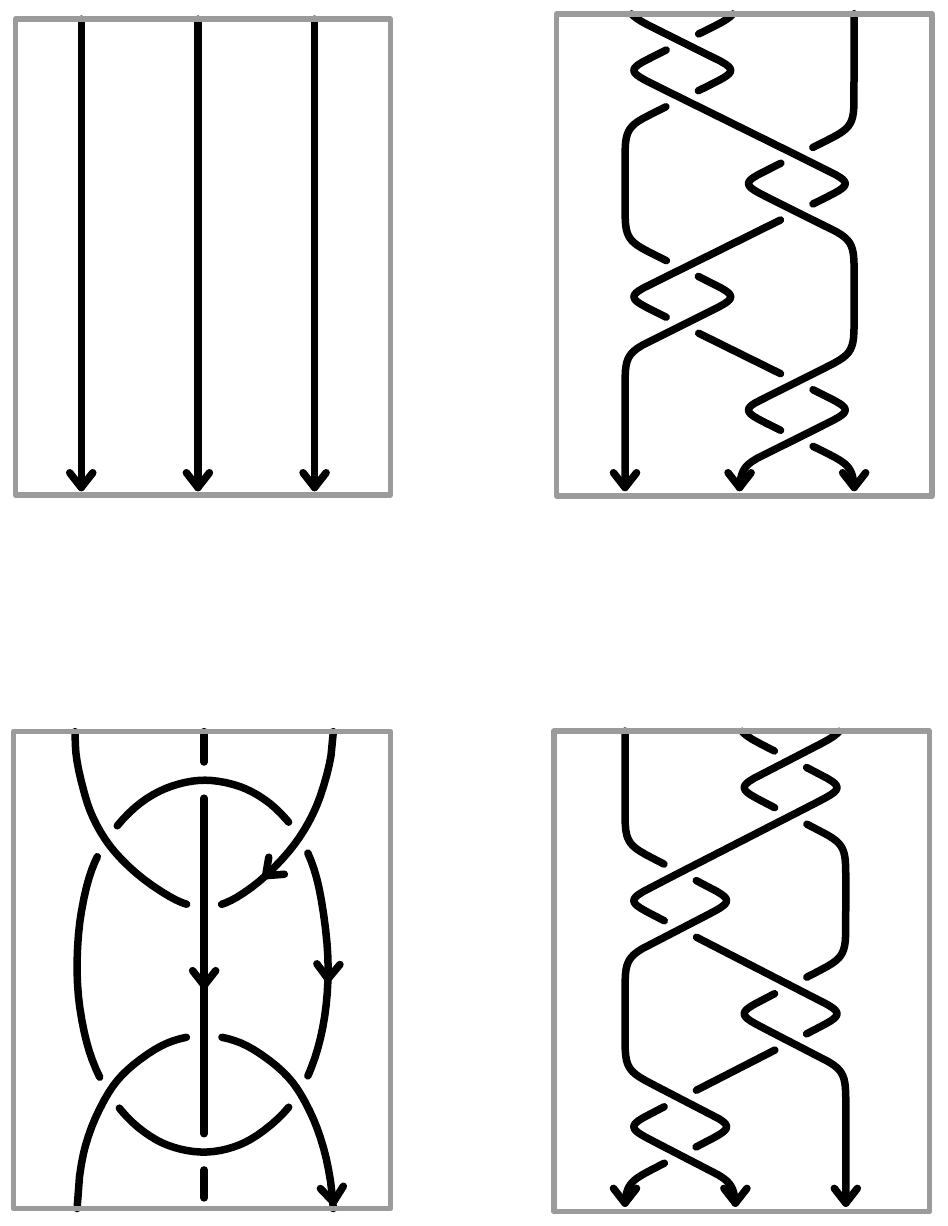} 
 \put(-270,175){$C_1$}\put(-270,150){$C_2$}\put(-270,125){$C_3$}\put(-120,175){$C_1$}\put(-120,150){$C_2$}\put(-120,125){$C_3$}\put(-270,64){$C_1$}\put(-270,38){$C_2$}\put(-270,13){$C_3$}\put(-120,70){$C_1$}\put(-70,82){$C_2$}\put(-120,40){$C_3$}
  \caption{Four examples distinguished by the triple linking number}
  \label{fig:3rd_linking} 
\end{figure}
\end{ex}

\begin{ex}\label{exa:4th-singly}
Let $C=C_1\cup C_2\cup C_3\cup C_4$ and $C'=C'_1\cup C'_2\cup C'_3\cup C'_4$ be the singly periodic tangles in Figure~\ref{fig:4th_linking} respectively. 
Their linking numbers vanish, and they have the same triple linking numbers, but they are distinguished by their 4th linking numbers. 
In the first case, for the very good motif illustrated in Figure~\ref{fig:4th_linking}, we have $\overline{\mu}_K(123)=2$ and $\overline{\mu}_K(124)=\overline{\mu}_K(134)=\overline{\mu}_K(234)=0$. 
Hence $\Delta_K(1234)=2$. 
Moreover, $\overline{\mu}_K(1234)=1$, and hence $\overline{\mu}_C(1234)=\frac{1}{2}$. 
In the second case, the triple linking numbers of the very good motif illustrated in Figure~\ref{fig:4th_linking} are the same as in the first case, and thus $\Delta_{K'}(1234)=2$. 
However, $\overline{\mu}_{K'}(1234)=0$, which gives $\overline{\mu}_{C'}(1234)=\frac{0}{2}=0$.
   \begin{figure}[h]
  \centering
  \includegraphics[angle=90, width=0.8\textwidth]{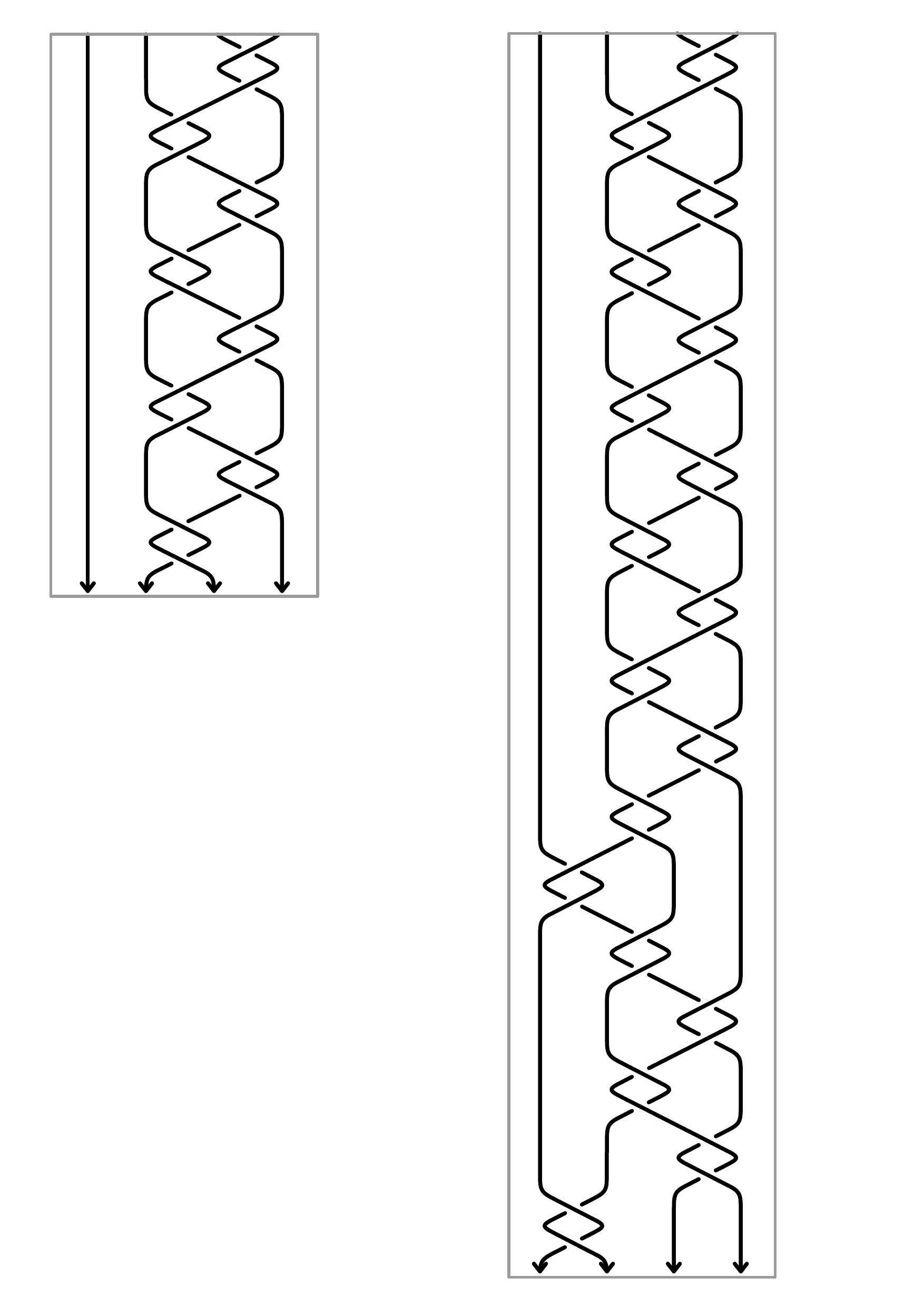} \put(-300,160){$C_1$}\put(-300,147){$C_2$}\put(-300,132){$C_3$}\put(-300,117){$C_4$}\put(-300,60){$C'_1$}\put(-300,45){$C'_2$}\put(-300,30){$C'_3$}\put(-300,15){$C'_4$}
  \caption{Two examples distinguished by the 4th linking number}
  \label{fig:4th_linking} 
\end{figure}
\end{ex}

\begin{thm}\label{thm:singly}
Let $C$ be the union of $m$ oriented components $C_1,\dots,C_m$ of a singly periodic tangle $\tau$. Then the higher order linking number $\overline{\mu}_C(J)$ does not depend on the choice of a very good motif. 
Moreover, if there is a link-homotopy from $\tau$ to a singly periodic tangle $\tau'$ that maps $C_i$ to components $C'_i$ of $\tau'$, respectively, then $\overline{\mu}_C(J)=\overline{\mu}_{C'}(J)$. 
In other words, the higher order linking number is well defined for a link-homotopy class of a singly periodic tangle.
\end{thm}

\begin{proof}
To show the first assertion, let $L'$ be the $n$-sheeted cover of a very good motif $L$ of $\tau$ for $n>1$ and let $F$ be a very good fundamental domain for $L$.
Then the union $F'$ of $n$ consecutive copies of $F$ is a very good fundamental domain for $L'$, since each open component meets only the two end
disks of $F'$.
Suppose that $K_i$ and $K'_i$ for $i=1, \dots, m$ are the images of $C_i$ in $L$ and $L'$, respectively. 
Let $K=K_1\cup\cdots\cup K_m$ and $K'=K'_1\cup\cdots\cup K'_m$.
Since any two very good motifs of $\tau$ have a common finite cover \cite[Proposition~2.4]{KMMY25}, and any finite cover of a very good motif is also very good, any two very good motifs have a common very good finite cover. 
Hence it suffices to show that $K$ and $K'$ give the same value of $\overline{\mu}_C(J)$.

Suppose that all $C_i$ are open.
Some of them may be oriented in the negative direction of $\mathbb{R}$.
Let $\widetilde{K}$ and $\widetilde{K}'$ be the links obtained from $K$ and $K'$ by reversing the orientations of the components $K_i$ and $K'_i$ for such $C_i$.
Reversing the orientation of a component changes $\overline{\mu}(J)$ only by sign for every non-repeating sequence $J$, and does not change $\Delta(J)$.
Moreover, the sign depends only on $J$ and the set of reversed components, which is common to $K$ and $K'$.
Hence, it suffices to show the assertion for $\widetilde{K}$ and $\widetilde{K}'$.
Since $L$ is very good, cutting $S^1\times D^2$ along the disk $D$ which is the image of $\partial I\times D^2\subset F$, which each component of $\widetilde{K}$ meets at one point, we obtain an $m$-component string link $\sigma$ whose closure is $\widetilde{K}$.
Then, $\widetilde{K}'$ is the closure of $\sigma^n$.
By the assumption in Definition~\ref{dfn:higher-singly} and induction on the length of $J'$, the indeterminacy $\Delta(J')$ vanishes and $\mu_\sigma(J')=0$ for all non-repeating sequences $J'$ of length at most $\lfloor\frac{m+1}{2}\rfloor$.
Since $2\lfloor\frac{m+1}{2}\rfloor\ge m$, applying Lemma~\ref{additive} to
$\sigma^n=\sigma\cdot\sigma^{n-1}$ inductively, we obtain $\mu_{\sigma^n}(J)=n\mu_{\sigma}(J)$ and $\Delta_{\sigma^n}(J)=n\Delta_{\sigma}(J)$ for any $J$ with length $\le m$.
Therefore
\[
  \overline{\mu}_{\widetilde{K}'}(J)=\mu_{\sigma^n}(J)\bmod\Delta_{\sigma^n}(J)
  =n\,\mu_\sigma(J)\bmod n\,\Delta_\sigma(J),
\]
and $\Delta_{\widetilde{K}'}(J)=n\,\Delta_{\widetilde{K}}(J)$.
If $\Delta_{\widetilde{K}}(J)\neq0$, this gives $\overline{\mu}_{\widetilde{K}'}(J)/\Delta_{\widetilde{K}'}(J)=\overline{\mu}_{\widetilde{K}}(J)/\Delta_{\widetilde{K}}(J)$ in $\mathbb{Q}/\mathbb{Z}$.
If $\Delta_{\widetilde{K}}(J)=0$, then $\Delta_{\widetilde{K}'}(J)=0$ and $\overline{\mu}_{\widetilde{K}'}(J)=n\,\overline{\mu}_{\widetilde{K}}(J)$, which has the same sign as $\overline{\mu}_{\widetilde{K}}(J)$.
In both cases, $K$ and $K'$ give the same value of $\overline{\mu}_C(J)$.
The conditions in Definition~\ref{dfn:higher-singly} hold for $K$ if and only if they hold for $K'$, since the first non-vanishing Milnor invariants are multiplied by $n$.

Otherwise, as above, we may assume that all open components among $C_1,\dots,C_m$ are oriented
in the positive direction of $\mathbb{R}$.
Since $L$ is very good, each open component of $K$ meets the disk $D$ which is the image of $\partial I\times D^2\subset F$ at one point, and the closed components are contained in the interior of $F$. 
Cutting $S^1\times D^2$ along $D$, we obtain an embedding $\sigma'\subset F=I\times D^2$ which is the union of a string link $\sigma_0$ and closed loops such that its closure is $K$; here the strings of $\sigma_0$ correspond to the open components, and the closed loops to the closed components.

By the assumption in Definition~\ref{dfn:higher-singly}, all Milnor invariants of the closure of
$\sigma_0$ with non-repeating sequences vanish.
By induction on the length of sequences, the indeterminacy vanishes, and hence the Milnor
invariants of $\sigma_0$ vanish as well.
Therefore $\sigma_0$ is link-homotopic to the trivial string link fixing the boundary~\cite{HL}.

Since each closed component of $K'$ is contained in a single copy of $F$ in $F'$, the link $K'$ is the closure of $\sigma'\cdot\sigma_0^{\,n-1}$. 
Moreover, $K'$ also satisfies the condition in Definition~\ref{dfn:higher-singly}, since the closure of $\sigma_0^n$ is link-homotopic to the trivial link.
Since $\sigma_0^{\,n-1}$ is link-homotopic to the trivial string link fixing the boundary,
$K'$ is link-homotopic to the closure of $\sigma'$, that is, to $K$.
Hence, $K$ and $K'$ have the same Milnor invariants, and they give the same value of $\overline{\mu}_C(J)$.

The second assertion follows from the first assertion and the invariance of Milnor invariants of links in $S^1\times D^2$ under link-homotopy, in the same way as Theorem~\ref{thm:lk_sp} by taking very good motifs of $\tau$ and $\tau'$ with the same period; note that the conditions in Definition~\ref{dfn:higher-singly} are also preserved.
This completes the proof.
\end{proof}

In the case of doubly periodic tangles, unless $C$ contains a pair of open components $C_i$ and $C_j$ with distinct directions, $C$ is contained in a singly periodic tangle in the same way as Proposition \ref{prop:sp_in_dp}. 
Therefore, in this case, the higher order linking number can also be defined. The same applies to triply periodic tangles.

\subsection{Triple linking numbers for doubly periodic tangles}\label{sec:Milnorfordp}

Let $C_{1}$, $C_{2}$, and $C_{3}$ be three oriented components of a doubly periodic tangle $\tau$.  
Suppose that there are at most two distinct directions among the open components of $C_{1}$, $C_{2}$, and $C_{3}$.  
The \emph{triple linking number} of $C_{1}$, $C_{2}$, and $C_{3}$, denoted by $\overline{\mu}_C$, is defined as follows. 
The existence of a very good motif follows from Lemma~\ref{lem:very_good_dp}. 

\begin{dfn}\label{def:Milnor_dp}
Let $L$ be a very good motif of a doubly periodic tangle $\tau$ for the union $C$ of three oriented components $C_1$, $C_2$, and $C_3$.
Suppose that $C_{1}$, $C_{2}$, and $C_{3}$ are respectively projected onto components $K_{1}$, $K_{2}$, and $K_{3}$ of $L$. 
Let $K=K_1\cup K_2 \cup K_3$.
We identify $H_1(T^2\times I;\mathbb{Z})$ with $\mathbb{Z}^2\subset\mathbb{R}^2$, where $T^2=\mathbb{R}^2/\mathbb{Z}^2$.
In general, we apply an orientation-preserving affine transformation of $\mathbb{R}^2$ preserving $\mathbb{Z}^2$ which maps $[K_i]$ to $(-1,0)$ and
$[K_j]$ to $(0,\pm1)$, and construct $\sigma$ as in Figures~\ref{Kandsigma-1} and~\ref{Kandsigma-3}.

(1) Suppose that $C_1$, $C_2$, and $C_3$ are all open, that $[K_j]=\pm[K_k]$, and that $\{[K_i],[K_j]\}$ is a basis of $H_1(T^2\times I;\mathbb{Z})$, as in Figure~\ref{Kandsigma-1}, where $\{i,j,k\}=\{1,2,3\}$ and $T$ is a tangle. 
\begin{figure}[ht]
  \centering
  \setlength{\unitlength}{0.0004\linewidth}
  \begin{picture}(2000,850)
    \put(0,0){\includegraphics[width=0.8\linewidth]{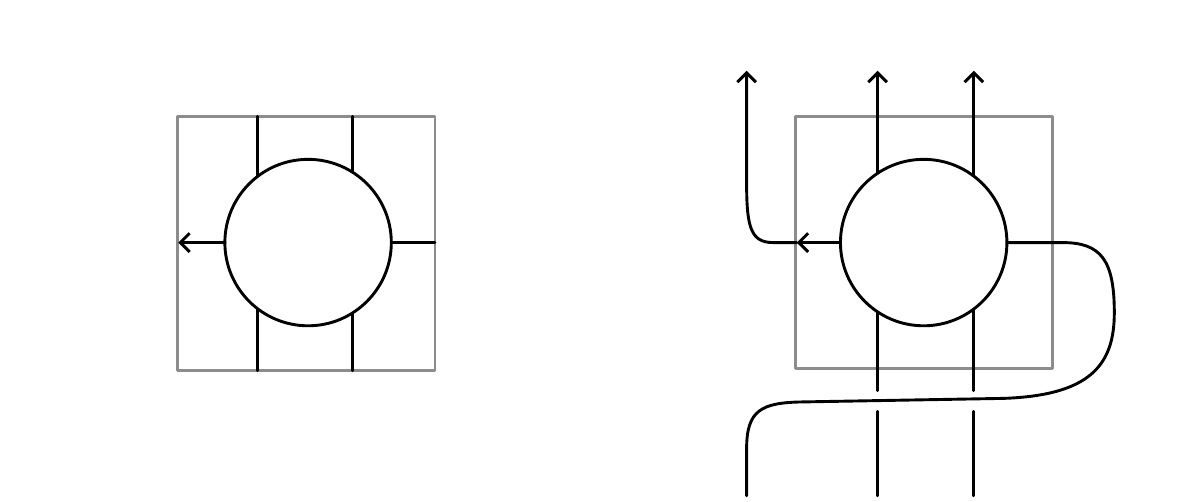}}
    \put(522,439){\makebox(0,0){\Huge $T$}}
    \put(437,708){\makebox(0,0){$K_j$}}
    \put(598,708){\makebox(0,0){$K_k$}}
    \put(437,170){\makebox(0,0){$K_j$}}
    \put(598,170){\makebox(0,0){$K_k$}}
    \put(290,439){\makebox(0,0)[r]{$L=K_i$}}
    \put(752,439){\makebox(0,0)[l]{$K_i$}}
    \put(1195,439){\makebox(0,0)[r]{$\sigma=$}}
    \put(1565,439){\makebox(0,0){\Huge $T$}}
    \put(1265,775){\makebox(0,0){$i$}}
    \put(1487,775){\makebox(0,0){$j$}}
    \put(1650,775){\makebox(0,0){$k$}}
  \end{picture}
  \caption{A very good motif $L$ and a string link $\sigma$ obtained from $K$}
  \label{Kandsigma-1}
\end{figure}
Then we define $\overline{\mu}_C$ as follows.
$$\overline{\mu}_C=\{\pm \mu_{\sigma}(ijk) \mod 
\gcd\{\operatorname{lk}_\sigma(i,j),\operatorname{lk}_\sigma(i,k)\}\},$$
where $\sigma$ is the 3-component string link obtained from $K$ as shown in Figure~\ref{Kandsigma-1}, whose strings are $i,j,k$ from left to right; it is constructed according to the orientation of $K_i$, and all of its strings are oriented upward, regardless of the orientations of $K_j$ and $K_k$.
Note that for $s=j,k$, $\mathrm{lk}_\sigma(i,s)=-\det([K_i]\ [K_s])\,\mathrm{lk}(K_i,K_s)+\tfrac12$, where the linking number $\mathrm{lk}(K_i,K_s)$ is defined in Subsection~\ref{subsec:Dp-tangle}.

(2) Suppose that $C_i$ and $C_j$ $(i<j)$ are open with distinct directions, that $C_k$ is closed, where $\{i,j,k\}=\{1,2,3\}$, and that $\{[K_i],[K_j]\}$ is a basis of $H_1(T^2\times I;\mathbb Z)$, as in Figure~\ref{Kandsigma-3}. 
\begin{figure}[ht]
  \centering
  \setlength{\unitlength}{0.0004\linewidth}
  \begin{picture}(2000,960)
    \put(0,0){\includegraphics[width=0.8\linewidth]{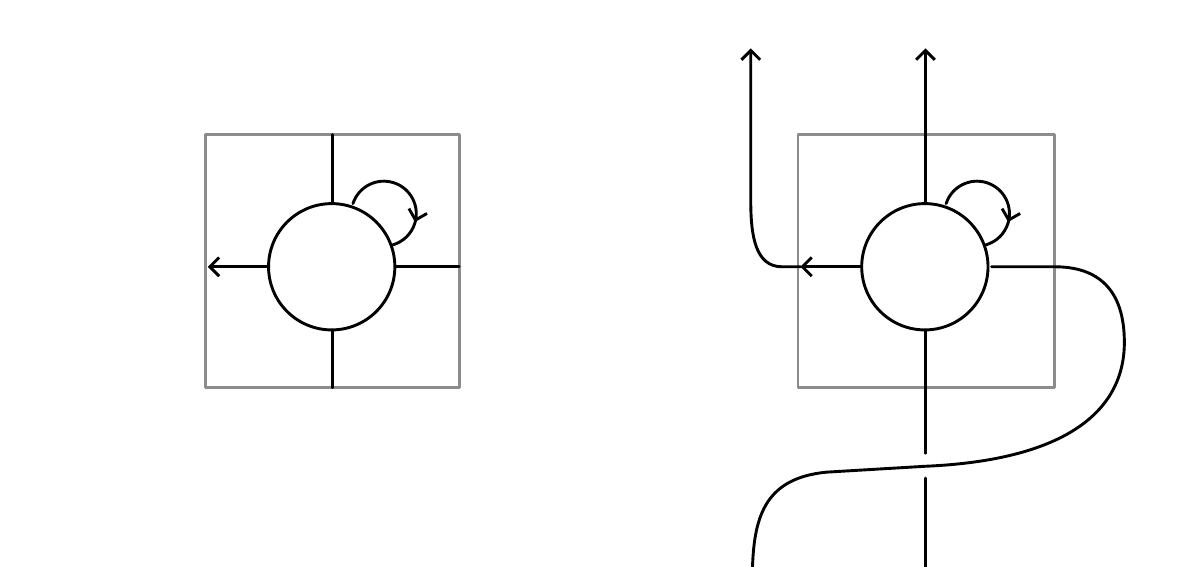}}
    \put(562,508){\makebox(0,0){\LARGE $T$}}
    \put(563,790){\makebox(0,0){$K_j$}}
    \put(563,245){\makebox(0,0){$K_j$}}
    \put(730,685){\makebox(0,0){$K_k$}}
    \put(335,508){\makebox(0,0)[r]{$L=K_i$}}
    \put(793,508){\makebox(0,0)[l]{$K_i$}}
    \put(1220,508){\makebox(0,0)[r]{$\sigma=$}}
    \put(1567,508){\makebox(0,0){\LARGE $T$}}
    \put(1272,930){\makebox(0,0){$i$}}
    \put(1578,930){\makebox(0,0){$j$}}
    \put(1730,685){\makebox(0,0){$k$}}
  \end{picture}
  \caption{A very good motif $L$ and a 2-component string link $\sigma$ with a closed loop from $K$}
  \label{Kandsigma-3}
\end{figure}
Then we define $\overline{\mu}_C$ as follows.
\[\overline{\mu}_C= -\det([K_i]\ [K_j])\mu_\sigma(ijk) \mod 
\gcd\{\operatorname{lk}_\sigma(i,k), \operatorname{lk}_\sigma(j,k)\},\]
where $\sigma$ is the union of an oriented 2-component string link, whose strings are $i,j$ from left to right, and a closed loop obtained from $K$ as shown in Figure~\ref{Kandsigma-3}; it is constructed according to the orientation of $K_i$, and both of its strings are oriented upward, regardless of the orientation of $K_j$.

(3) If the open components of $C$ have at most one distinct direction, then, for each non-repeating sequence $J$ on $\{1,2,3\}$, we define $\overline{\mu}_C(J)$ as in the singly periodic case, provided that the assumptions of Definition~\ref{dfn:higher-singly} are satisfied.
\end{dfn}

In cases (1) and (2), $\overline{\mu}_C$ is a single invariant corresponding to sequences of length three, whereas in case (3) an invariant $\overline{\mu}_C(J)$ is defined for each non-repeating sequence $J$.
Note that in case (1), $\overline{\mu}_C$ is defined up to sign, and hence it does not depend on the orientations of the components $C_i$. 
Moreover, it is invariant under the mirror image with respect to a line in $T^2$.
In case (2), we have $\overline{\mu}_{C'}=-\overline{\mu}_C$ if $C'$ is obtained from $C$ by changing the orientation of $C_j$ or $C_k$, while $\overline{\mu}_{C'}=\overline{\mu}_C$ if $C'$ is obtained from $C$ by changing the orientation of $C_i$. 
Moreover, in case (2), $\overline{\mu}_C$ is invariant
under the mirror images with respect to a line in $T^2$ and to $T^2\times\{\frac12\}$.
In all cases, $\overline{\mu}_C$ is invariant under orientation-preserving affine transformations of $\mathbb{R}^2\times I$ of the form $(x,z)\mapsto(Ax+b,z)$ with $A\in SL(2,\mathbb{Z})$, such as the rotation by $\pi/2$.

\begin{ex}
For integers $a,b,c,d$, by Lemma~\ref{standcal}, the triple linking number $\overline\mu_C$ of the very good motif shown in  Figure~\ref{generalex_triplelinking} takes the value $\pm d \bmod \gcd\{a,b\}$, which is independent of $c$. This falls under case~(1).
\begin{figure}[ht]
  \centering
 \setlength{\unitlength}{\dimexpr0.26\linewidth/910\relax}
  \begin{picture}(910,870)
    \put(0,0){\includegraphics[width=910\unitlength]{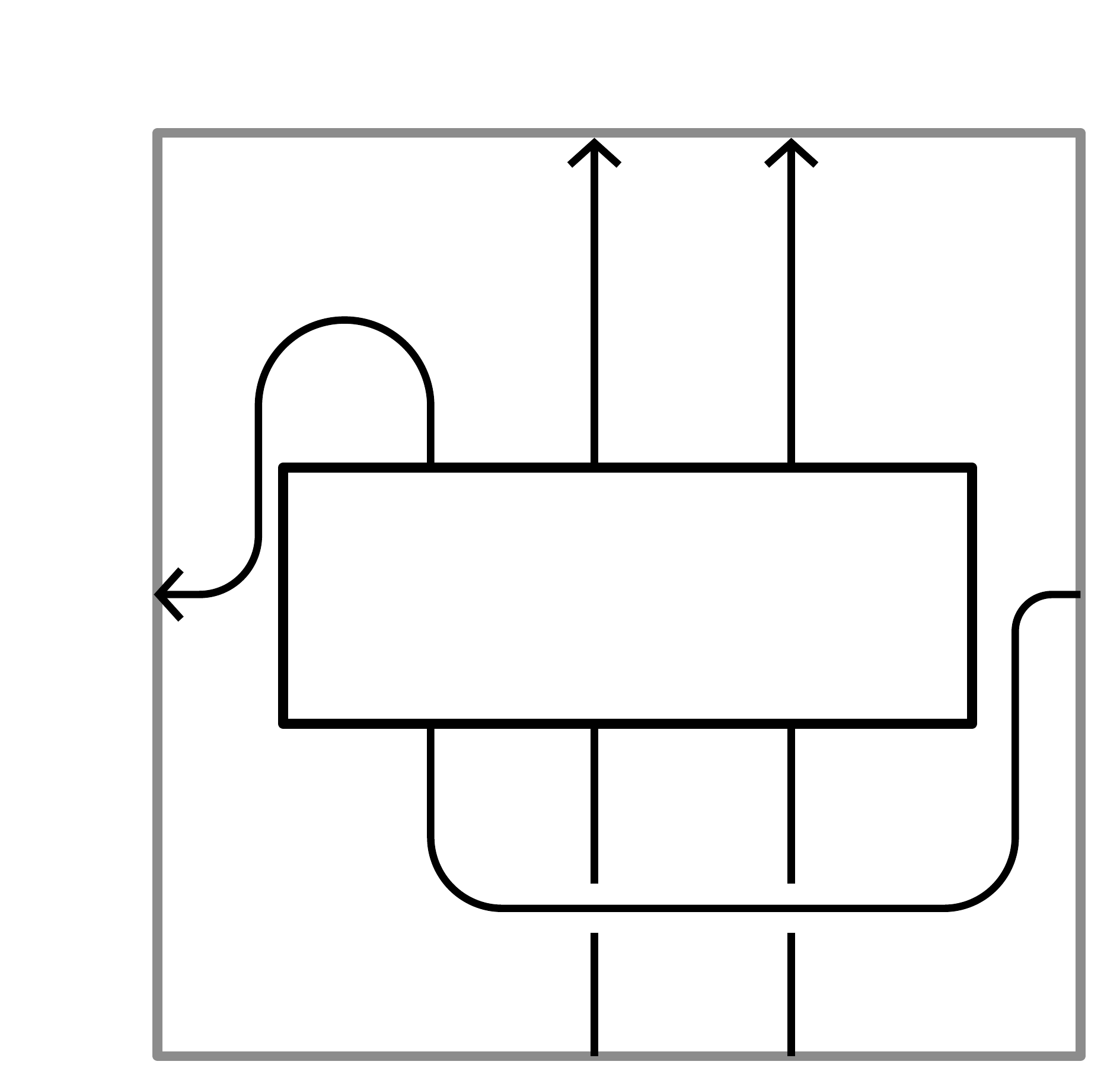}}
    \put(483,819){\makebox(0,0){$K_2$}}
    \put(643,819){\makebox(0,0){$K_3$}}
    \put(110,387){\makebox(0,0)[r]{$K_1$}}
    \put(510,386){\makebox(0,0){$\sigma(a,b,c,d)$}}
  \end{picture}
  \caption{Example of the triple linking number in case~(1)}
  \label{generalex_triplelinking}
\end{figure}
 
In particular, for $(a,b,c,d)=(0,0,0,0)$ and $(0,0,0,1)$, we obtain the very good motifs shown in Figure~\ref{triplelinking_example}, whose triple linking numbers $\overline\mu_C$ are $0$ and $\pm1$ in $\mathbb{Z}$, respectively, since $\gcd\{0,0\}=0$.
\begin{figure}[ht]
  \centering
  \setlength{\unitlength}{\dimexpr\linewidth/3040\relax}
  \begin{picture}(3040,690)
    \put(110,0){\includegraphics[width=2930\unitlength]{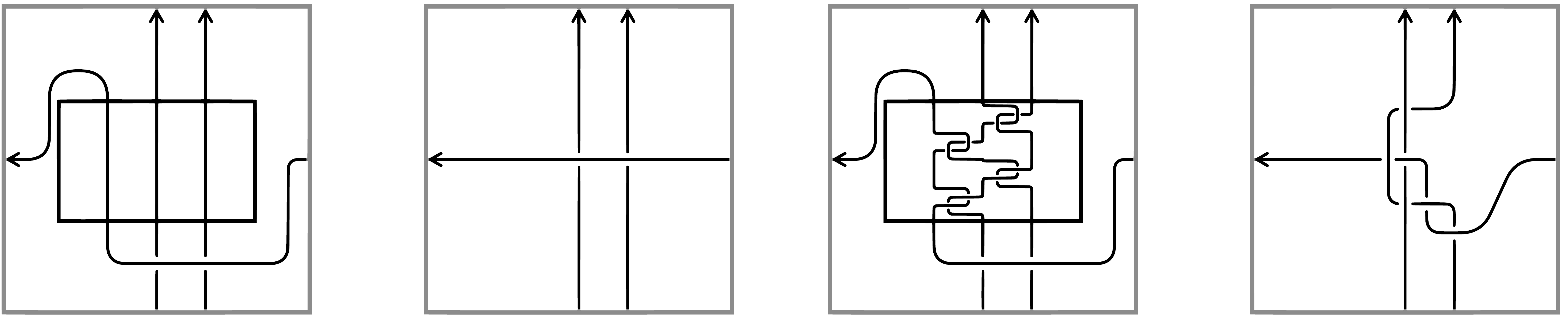}}
    \put(384,648){\makebox(0,0){$K_2$}}
    \put(515,648){\makebox(0,0){$K_3$}}
    \put(108,301){\makebox(0,0)[r]{$K_1$}}
    \put(797,301){\makebox(0,0){\Large $=$}}
    \put(1928,648){\makebox(0,0){$K_2$}}
    \put(2059,648){\makebox(0,0){$K_3$}}
    \put(1652,301){\makebox(0,0)[r]{$K_1$}}
    \put(2341,301){\makebox(0,0){\Large $=$}}
  \end{picture}
  \caption{The cases $(a,b,c,d)=(0,0,0,0)$ and $(0,0,0,1)$ in case~(1)}
  \label{triplelinking_example}
\end{figure}
\end{ex}

\begin{ex}
For integers $a,b,c,d$, by Lemma~\ref{standcal}, the triple linking number $\overline\mu_C$ of the very good motif shown in Figure~\ref{generalex_triplelinking2} takes the value $d \bmod \gcd\{b,c\}$, which is independent of $a$. This falls under case~(2).

\begin{figure}[ht]
  \centering
  \setlength{\unitlength}{\dimexpr0.26\linewidth/910\relax}
  \begin{picture}(910,870)
    \put(0,0){\includegraphics[width=910\unitlength]{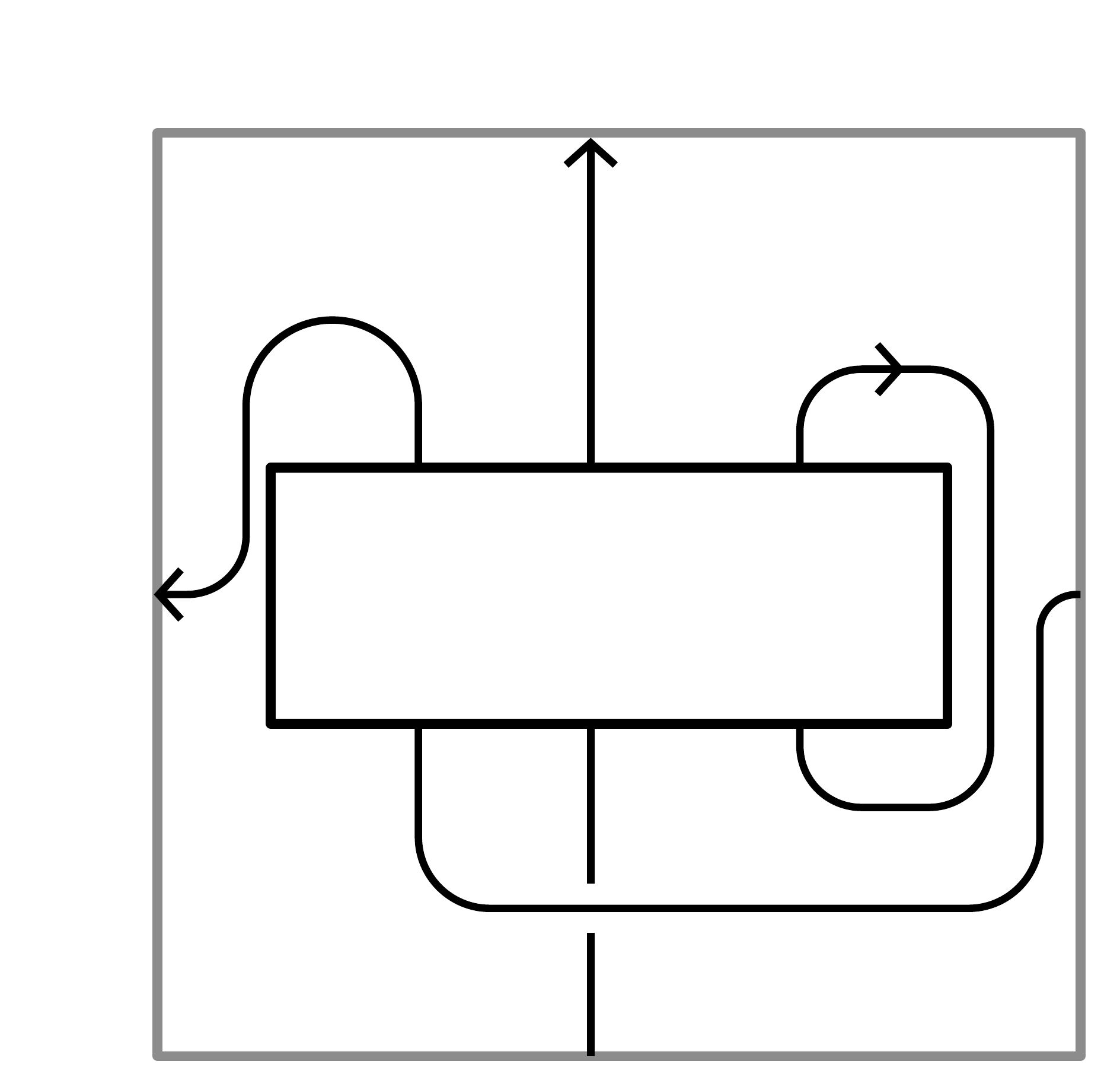}}
    \put(480,815){\makebox(0,0){$K_2$}}
    \put(790,660){\makebox(0,0){$K_3$}}
    \put(110,387){\makebox(0,0)[r]{$K_1$}}
    \put(495,386){\makebox(0,0){$\sigma(a,b,c,d)$}}
  \end{picture}
  \caption{Example of the triple linking number in case~(2)}
  \label{generalex_triplelinking2}
\end{figure}
In particular, for $(a,b,c,d)=(0,0,0,0)$ and $(0,0,0,1)$, we obtain the very good motifs shown in Figure~\ref{triplelinking_example2}, whose triple linking numbers $\overline\mu_C$ are $0$ and $1$ in $\mathbb{Z}$, respectively, since $\gcd\{0,0\}=0$.
\begin{figure}[ht]
  \centering
  \setlength{\unitlength}{\dimexpr\linewidth/3065\relax}
  \begin{picture}(3065,617)
    \put(0,0){\includegraphics[width=\linewidth]{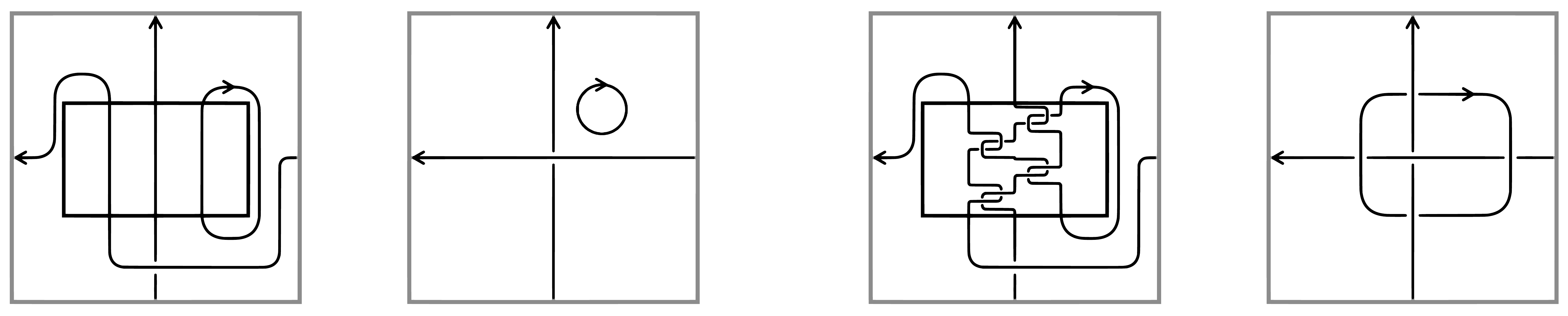}}
    \put(495,520){\makebox(0,0){$K_3$}}
    \put(1290,535){\makebox(0,0){$K_3$}}
    \put(693,308){\makebox(0,0){\Large $=$}}
    \put(2170,520){\makebox(0,0){$K_3$}}
    \put(2930,520){\makebox(0,0){$K_3$}}
    \put(2373,308){\makebox(0,0){\Large $=$}}
  \end{picture}
  \caption{The cases $(a,b,c,d)=(0,0,0,0)$ and $(0,0,0,1)$ in case~(2)}
  \label{triplelinking_example2}
\end{figure}
\end{ex}

\begin{thm}\label{thm:triple_dp}
Let $C$ be the union of three oriented components $C_1,C_2,C_3$ of a doubly periodic tangle $\tau$ as in Definition~\ref{def:Milnor_dp}.
Then the triple linking number $\overline{\mu}_C$ does not depend on the choice of a very good motif. 
Moreover, if there is a link-homotopy from $\tau$ to a doubly periodic tangle $\tau'$ that maps $C_i$ to $C'_i$ for $i=1,2,3$, then $\overline{\mu}_C=\overline{\mu}_{C'}$, where $C'=C'_1\cup C'_2\cup C'_3$. 
In other words, the triple linking number is well-defined for a link-homotopy class of a doubly periodic tangle.
\end{thm}

\begin{proof}
We first show that $\overline{\mu}_C$ does not depend on the choice of a very good fundamental domain $F$ for a fixed very good motif $L$.
In cases (1) and (2), since the opposite pairs of edges of $P$ are glued via the translations by $[K_i]$ and $[K_j]$, such $F$ is determined up to translation.
When $F$ is translated, the string link $\sigma$ changes by an isotopy fixing the boundary, except when a crossing of two parallel components passes through the top or bottom face of $F$, which gives the move shown in Figure~\ref{conjofhalftwist}.
This move occurs only in case (1).

\begin{figure}[ht]
  \centering
  \setlength{\unitlength}{0.004\linewidth}
  \begin{picture}(162,54)
    \put(0,0){\includegraphics[width=0.648\linewidth]{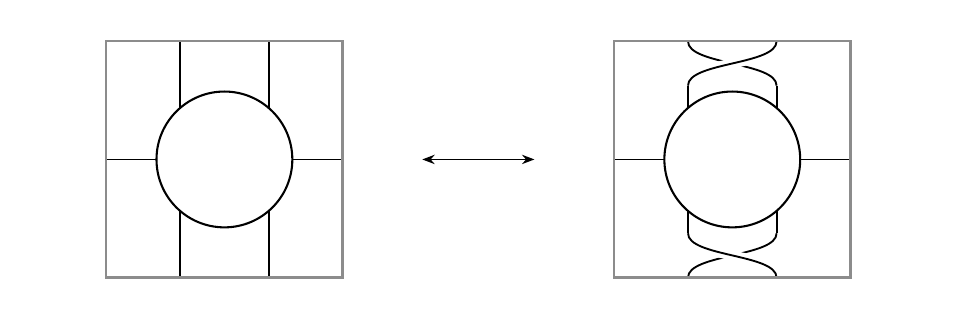}}
    \put(38,27){\makebox(0,0){\Large $T$}}
    \put(30.5,51.5){\makebox(0,0){$K_j$}}
    \put(45.5,51.5){\makebox(0,0){$K_k$}}
    \put(30.5,2.5){\makebox(0,0){$K_j$}}
    \put(45.5,2.5){\makebox(0,0){$K_k$}}
    \put(16.6,27){\makebox(0,0)[r]{$K_i$}}
    \put(59.4,27){\makebox(0,0)[l]{$K_i$}}
    \put(124,27){\makebox(0,0){\Large $T$}}
    \put(116.5,51.5){\makebox(0,0){$K_k$}}
    \put(131.5,51.5){\makebox(0,0){$K_j$}}
    \put(116.5,2.5){\makebox(0,0){$K_k$}}
    \put(131.5,2.5){\makebox(0,0){$K_j$}}
    \put(102.6,27){\makebox(0,0)[r]{$K_i$}}
    \put(145.4,27){\makebox(0,0)[l]{$K_i$}}
  \end{picture}
  \caption{Conjugation by a half twist}
  \label{conjofhalftwist}
\end{figure}

Let $F$ and $F'$ be very good fundamental domains for $L$ such that $K$ with $F$ and $K$ with $F'$ are related by the move shown in Figure~\ref{conjofhalftwist}, and let $\sigma$ and $\sigma'$ be the string links obtained from $K$ with $F$ and $F'$, respectively.
Then $\sigma'$ is obtained from $\sigma$ by conjugating with a half twist of the $j$th and $k$th strings, and the linking numbers between the components do not change.
Moreover, by Lemma~\ref{half-twist-mu123},
\[
\mu_{\sigma'}(ikj)\equiv-\mu_\sigma(ijk)
\pmod{\gcd\{\mathrm{lk}_\sigma(i,j),\mathrm{lk}_\sigma(i,k)\}}.
\]
Since $\overline{\mu}_C$ is defined up to sign, it does not change under the move.

To show the first assertion, let $L'$ be a very good finite cover of a very good motif $L$ of $\tau$. 
Suppose that $K_i$ and $K_i'$ for $i = 1, 2, 3$ are the images of $C_i$ respectively in $L$ and $L'$. 
Let $F$ be a good fundamental domain for $L$. 
The knot $K_i'$ is a component of the preimage of $K_i$ by the finite covering map. 
Let $K'=K'_1\cup K'_2\cup K'_3$.

We first consider case (1). 
The cover $L'$ is obtained by combining covers of types (H) and (V) shown in
Figure~\ref{n-cov}, with $n$, $m$, and the fundamental domain $F'$ as in the
figure.
\begin{figure}[ht]
  \centering
  \setlength{\unitlength}{\dimexpr\linewidth/634\relax}
  \begin{picture}(634,410)
    \put(0,0){\includegraphics[width=634\unitlength]{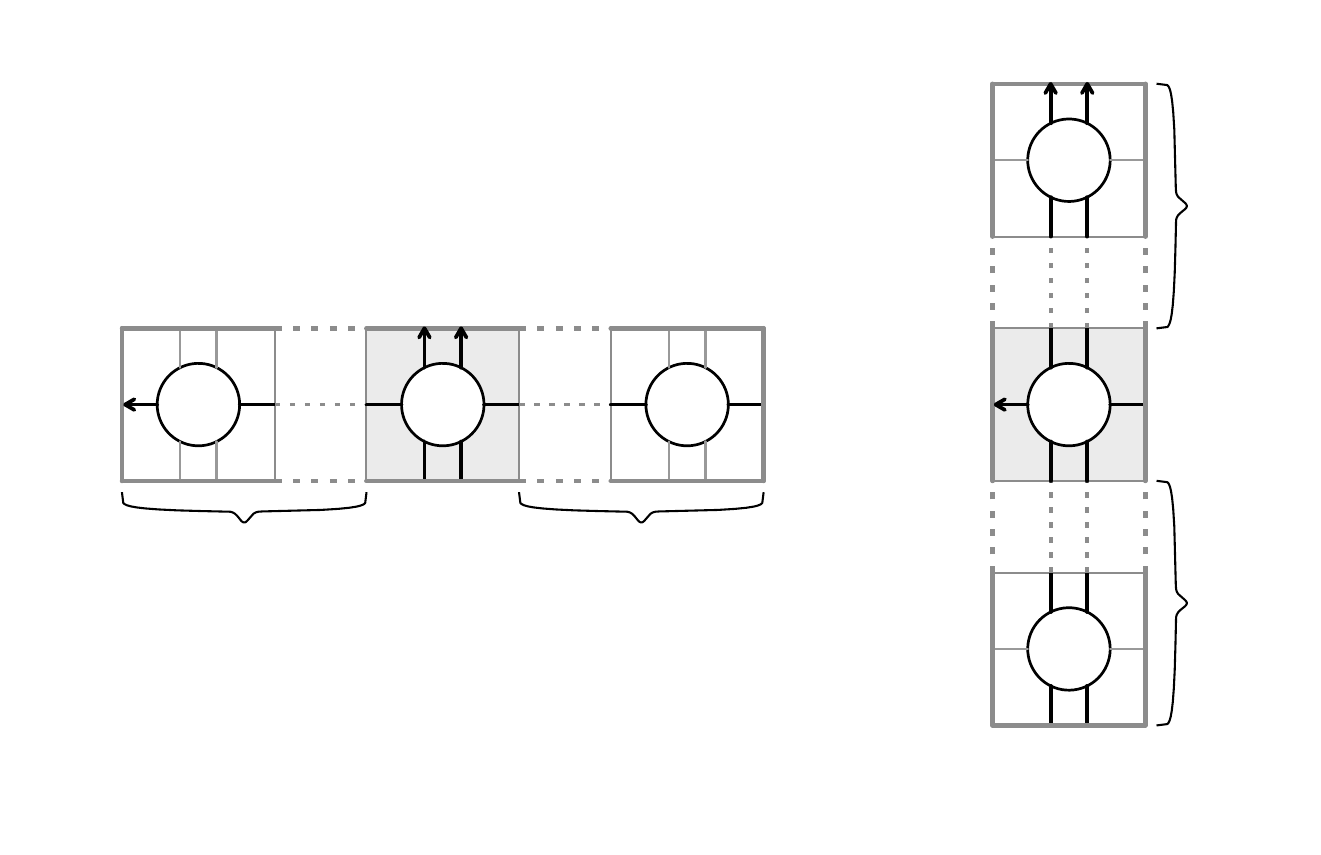}}
    \put(95.0,216.7){\makebox(0,0){$T$}}
    \put(212.4,216.7){\makebox(0,0){$T$}}
    \put(329.8,216.7){\makebox(0,0){$T$}}
    \put(56,216.7){\makebox(0,0)[r]{$K_i'$}}
    \put(206.5,270){\makebox(0,0)[r]{$K_j'$}}
    \put(218.3,270){\makebox(0,0)[l]{$K_k'$}}
    \put(182,184.5){\makebox(0,0)[bl]{\small $F$}}
    \put(117.4,150){\makebox(0,0){$m-1$}}
    \put(307.4,150){\makebox(0,0){$n-m$}}
    \put(211.5,121){\makebox(0,0){(H)}}
    \put(512.6,334.1){\makebox(0,0){$T$}}
    \put(512.6,216.7){\makebox(0,0){$T$}}
    \put(512.6,99.4){\makebox(0,0){$T$}}
    \put(473.5,216.7){\makebox(0,0)[r]{$K_i'$}}
    \put(506.7,388){\makebox(0,0)[r]{$K_j'$}}
    \put(518.6,388){\makebox(0,0)[l]{$K_k'$}}
    \put(482.5,184.5){\makebox(0,0)[bl]{\small $F$}}
    \put(576,313.2){\makebox(0,0)[l]{$n-m$}}
    \put(576,122.4){\makebox(0,0)[l]{$m-1$}}
    \put(512.5,26){\makebox(0,0){(V)}}
  \end{picture}
  \caption{The covers of types (H) and (V) of $L$}
  \label{n-cov}
\end{figure}
Let $\sigma'$ be the string link obtained from $K'$ and $F'$. 
We show that, for a cover of type (H) or (V),
\begin{equation}\label{eq:case1-cover}
\begin{aligned}
&\mathrm{lk}_{\sigma'}(i,s)=\mathrm{lk}_\sigma(i,s)\quad(s=j,k),\\
&\mu_{\sigma'}(ijk)\equiv\mu_\sigma(ijk)
\pmod{\gcd\{\mathrm{lk}_\sigma(i,j),\mathrm{lk}_\sigma(i,k)\}}.
\end{aligned}
\end{equation}
Since the modulus is preserved at each step, \eqref{eq:case1-cover} then holds for any very good finite cover $L'$.

Suppose that $L'$ is of type (H). 
Then $K'_i$ passes through all the copies of $F$ in $F'$, while $K'_j$ and $K'_k$ meet only the copy $F$. 
The other copies contain only the string coming from $K'_i$. 
Hence $\sigma'$ is link-homotopic to $\sigma$ fixing the boundary, and \eqref{eq:case1-cover} follows.

Suppose that $L'$ is of type (V). Then $K'_j$ and $K'_k$ pass through all the copies of $F$ in $F'$, while $K'_i$ meets only the copy $F$. 
Each of the other $n-1$ copies gives the 2-component string link obtained from $\sigma$ by deleting its $i$th string, which is link-homotopic to $x^{l}$, where $l=\mathrm{lk}_\sigma(j,k)$ and $x$ denotes the positive full twist of the $j$th and $k$th strings; indeed,
2-component string links are classified up to link-homotopy by the linking number. 
The rerouted part of the $i$th string runs below $F'$ and crosses the $j$th and $k$th strings simultaneously, so that the full twists coming from the $m-1$ copies below $F$ can be moved below it. 
Therefore $\sigma'$ is link-homotopic to
\[
x^{(m-1)l}\cdot\sigma\cdot x^{(n-m)l}
\]
fixing the boundary.
By Lemma~\ref{product-mu123}, the full twists attached below do not change $\mu(ijk)$, and those attached above change it by a multiple of
$\mathrm{lk}_\sigma(i,k)-\mathrm{lk}_\sigma(i,j)$; more precisely,
$\mathrm{lk}_{\sigma'}(i,s)=\mathrm{lk}_\sigma(i,s)$ for $s=j,k$,
$\mathrm{lk}_{\sigma'}(j,k)=nl$, and
\[
\mu_{\sigma'}(ijk)=\mu_\sigma(ijk)
+(n-m)\,l\,\bigl(\mathrm{lk}_\sigma(i,k)-\mathrm{lk}_\sigma(i,j)\bigr),
\]
which implies \eqref{eq:case1-cover}.

Since any two very good motifs of $\tau$ for $C_1,C_2,C_3$ have a common very good finite cover, the triple linking number $\overline{\mu}_C$ is well-defined in case (1).

We next consider case (2). 
Since $\{[K_i],[K_j]\}$ is a basis of $H_1(T^2\times I;\mathbb{Z})$, any very good finite cover $L'$ of $L$ is obtained by successively taking covers in the directions of $[K_i]$ and $[K_j]$.
For each of them, as in the case of type (H) above, the other copies of the fundamental domain contain only one string after removing the strings and the closed loops not coming from $K'_i,K'_j,K'_k$, where $K'_k$ is contained in the copy $F$.
Hence, the resulting union $\sigma'$ of a 2-component string link and a closed loop is link-homotopic to $\sigma$ fixing the endpoints of the $i$th and $j$th components, and thus $\overline{\mu}_{\sigma'}(ijk)=\overline{\mu}_\sigma(ijk)$ by
Proposition~\ref{prop:case3-invariance}. 
Moreover, the linking numbers $\mathrm{lk}_\sigma(i,k)$ and $\mathrm{lk}_\sigma(j,k)$ do not change, and neither does the sign of $\det([K_i]\ [K_j])$, since $[K'_i]$ and $[K'_j]$ are positive multiples of $[K_i]$ and $[K_j]$. 
Therefore, the triple linking number $\overline{\mu}_C$ is well-defined in case (2).

Now suppose that $\tau$ and $\tau'$ are related by a link-homotopy as in the theorem. Taking a common very good motif, we may regard it as a link-homotopy of motifs in $T^2\times I$.
Such a link-homotopy is a composition of ambient isotopies and self-crossing changes, and we may assume that each self-crossing change takes place in the interior of the fundamental domain $F$. 
Then the string link $\sigma$ changes by a link-homotopy fixing the boundary, under which $\mu_\sigma(ijk)$ is invariant in case (1), and
$\overline{\mu}_\sigma(ijk)$ is invariant in case (2) by Proposition~~\ref{prop:case3-invariance}, or $F$ is translated, which does not change $\overline{\mu}_C$ as shown above.
This completes the proof.
\end{proof}

\subsection{Triple linking numbers for triply periodic tangles}

Let $C$ be the union of three oriented components $C_{1}$, $C_{2}$, $C_{3}$ of a triply periodic tangle $\tau$. 
Suppose that there are at most two distinct directions among the open components of $C$, or that there are three distinct directions that are linearly independent. 
The \emph{triple linking number} of $C$, denoted by $\overline{\mu}_C$, is defined as follows. 
The existence of a very good motif follows from Lemma~\ref{lem:good_tp}. 

\begin{dfn}\label{def:triple}
Let $L$ be a very good motif of a triply periodic tangle $\tau$ for the union $C$ of three oriented components $C_1$, $C_2$, and $C_3$.
Suppose that $C_{1}$, $C_{2}$, and $C_{3}$ are respectively projected onto components $K_{1}$, $K_{2}$, and $K_{3}$ of $L$. 
Let $K=K_1\cup K_2\cup K_3$.
If $[K_1],[K_2],[K_3]$ form a basis of $H_1(T^3;\mathbb{Z})$, as in Figure~\ref{Kandsigma-4}, 
then, by identifying $H_1(T^3;\mathbb{Z})$ with $\mathbb{Z}^3\subset\mathbb{R}^3$ where $T^3=\mathbb{R}^3/\mathbb{Z}^3$, we define $\overline{\mu}_C$ as follows.
\[\overline{\mu}_C= -\det([K_1]\ [K_2]\ [K_3])\mu_{\sigma}(123)\in\mathbb{Z},\]
where $\sigma$ is the 3-component string link obtained from $K$ as shown in Figure~\ref{Kandsigma-4}, whose strings are $1,2,3$ from left to right; it is constructed according to the orientations of $K_1$ and $K_2$, and all of its strings are oriented upward, regardless of the orientation of $K_3$.
\begin{figure}[ht]
    \centering
    \includegraphics[width=0.8\linewidth]{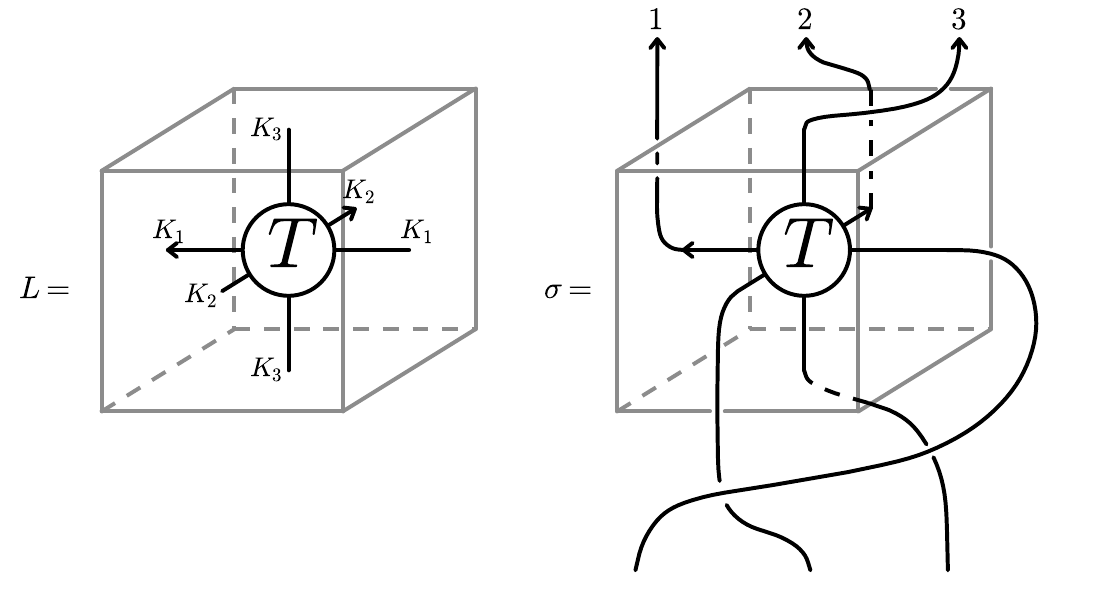}
    \caption{A very good motif $L$ and a string link $\sigma$ obtained from $K$}
    \label{Kandsigma-4}
\end{figure}

If the open components of $C$ have at most one distinct direction, we define $\overline\mu_C(J)$ for each non-repeating sequence $J$ on $\{1,2,3\}$ as in the singly periodic case, provided that the assumptions of Definition~\ref{dfn:higher-singly} are satisfied. If they have exactly two distinct directions, we define $\overline\mu_C$ as in the doubly periodic case (Definition~\ref{def:Milnor_dp}).
\end{dfn}

\begin{thm}\label{thm:triple_tp}
Let $C$ be the union of three oriented components $C_1,C_2,C_3$ of a triply periodic tangle $\tau$ as in Definition~\ref{def:triple}. 
Then the triple linking number $\overline{\mu}_C$ does not depend on the choice of a very good motif.
Moreover, if there is a link-homotopy from $\tau$ to a triply periodic tangle $\tau'$ that maps $C_i$ to $C'_i$ for $i=1,2,3$, then $\overline{\mu}_C=\overline{\mu}_{C'}$, where $C'=C'_1\cup C'_2\cup C'_3$. 
In other words, the triple linking number is well-defined for a link-homotopy class of a triply periodic tangle.
\end{thm}

\begin{proof}
The proof is similar to that of Theorem~\ref{thm:triple_dp}.
For the first assertion, any very good finite cover $L'$ of a very good motif $L$ is obtained by successively taking covers in the directions of $[K_1]$, $[K_2]$ and $[K_3]$. 
For each of them, as in the case of type (H) in the proof of Theorem~\ref{thm:triple_dp}, the other copies of the fundamental domain contain only one string after removing the strings not coming from $K'_1,K'_2,K'_3$.
Hence the resulting string link is link-homotopic to $\sigma$ fixing the boundary, and $\mu_\sigma(123)$ does not change. 
Moreover, the sign of $\det([K_1]\,[K_2]\,[K_3])$ does not change, since $[K'_s]$ is a positive multiple of $[K_s]$.
For the second assertion, we argue as in the proof of Theorem~\ref{thm:triple_dp}. 
Note that no move as in Figure~\ref{conjofhalftwist} occurs, since each face of the fundamental domain meets only one of $K_1$, $K_2$ and $K_3$.
In the other cases, the assertions follow from Theorems~\ref{thm:singly} and \ref{thm:triple_dp}.
\end{proof}

\begin{ex}
For integers $a,b,c,d$, by Lemma~\ref{standcal}, the triple linking number $\overline\mu_C$ of the very good motif shown in Figure~\ref{fig:generalex_triplelinking3} takes the integer value $ab+bc-ca+d$.

\begin{figure}[ht]
  \centering
  \setlength{\unitlength}{\dimexpr0.4\linewidth/246\relax}
  \begin{picture}(246,212)
    \put(0,0){\includegraphics[width=245.7\unitlength]{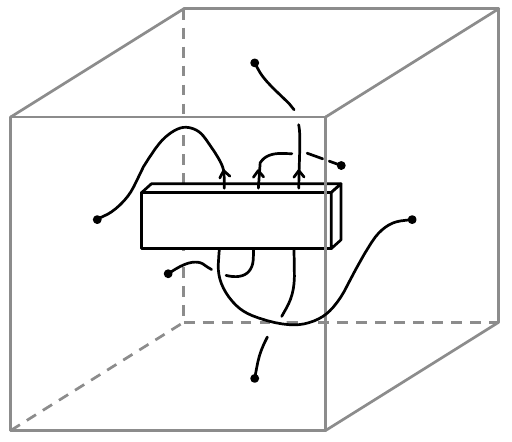}}
    \put(113.5,105.9){\makebox(0,0){$\sigma(a,b,c,d)$}}
    \put(42,106.4){\makebox(0,0)[r]{$K_1$}}
    \put(202,106.4){\makebox(0,0)[l]{$K_1$}}
    \put(76.5,80.4){\makebox(0,0)[r]{$K_2$}}
    \put(168,132.3){\makebox(0,0)[l]{$K_2$}}
    \put(126.5,30.1){\makebox(0,0)[l]{$K_3$}}
    \put(126.5,181.6){\makebox(0,0)[l]{$K_3$}}
  \end{picture}
  \caption{Example of the triple linking number}
  \label{fig:generalex_triplelinking3}
\end{figure}
 
In particular, for $(a,b,c,d)=(0,0,0,0)$ and $(0,0,0,1)$, we obtain the very good motifs shown in Figure~\ref{fig:triplelinking_example3}, whose triple linking numbers $\overline\mu_C$ are $0$ and $1$ in $\mathbb{Z}$, respectively.
    \begin{figure}[ht]
  \centering
  \setlength{\unitlength}{\dimexpr0.9\linewidth/880\relax}
  \begin{picture}(880,212)
    \put(0,0){\includegraphics[width=879.26\unitlength]{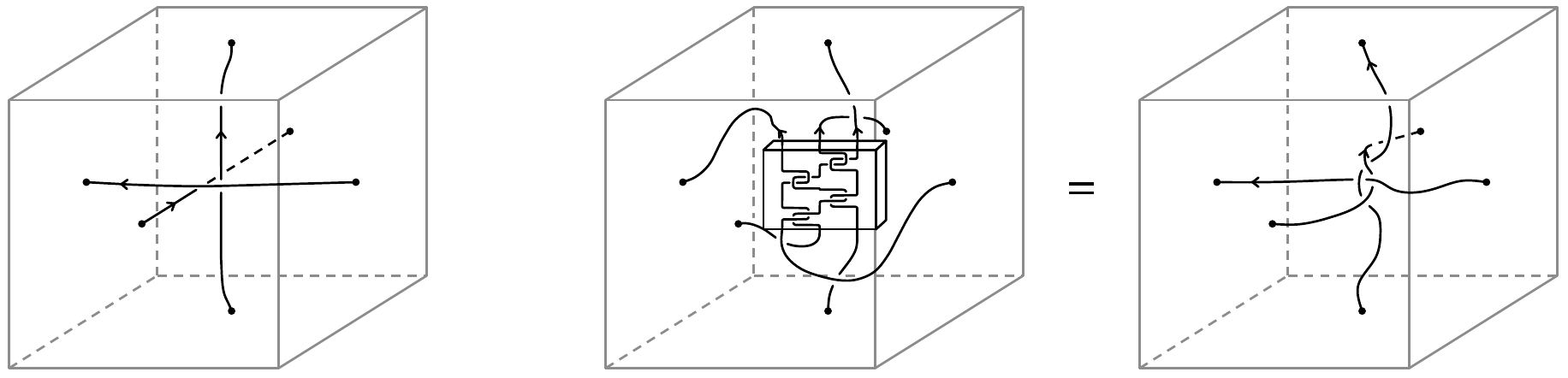}}
    \put(44,109.5){\makebox(0,0)[r]{$K_1$}}
    \put(135,188){\makebox(0,0)[l]{$K_3$}}
    \put(167,138){\makebox(0,0)[l]{$K_2$}}
    \put(378.5,109.5){\makebox(0,0)[r]{$K_1$}}
    \put(469.5,188){\makebox(0,0)[l]{$K_3$}}
    \put(501.5,138){\makebox(0,0)[l]{$K_2$}}
    \put(678,109.5){\makebox(0,0)[r]{$K_1$}}
    \put(769,188){\makebox(0,0)[l]{$K_3$}}
    \put(801,138){\makebox(0,0)[l]{$K_2$}}
  \end{picture}
  \caption{The cases $(a, b, c, d)=(0, 0, 0, 0)$ and $(0, 0, 0, 1)$}
  \label{fig:triplelinking_example3}
\end{figure}
\end{ex}


\section*{Acknowledgements} 
This work was supported by the World Premier International Research Center Initiative Program, International Institute for Sustainability with Knotted Chiral Meta Matter (WPI-SKCM$^2$), MEXT, Japan, 
and JSPS Program for Forming Japan's Peak Research Universities (J-PEAKS) Grant Number JPJS00420230011. 
The first author was supported by JSPS KAKENHI Grant Number 25K07005.

\bibliographystyle{plain}
\bibliography{ref-lnpt}

\end{document}